\documentclass[12pt, a4paper,twoside,reqno]{amsproc}
\usepackage[utf8]{inputenc}
\usepackage[T1]{fontenc}
\usepackage[utf8]{inputenc}

\usepackage[dvipsnames]{xcolor}
\usepackage{graphicx}
\usepackage{float}
\usepackage{placeins}
\usepackage{subcaption}
\usepackage{array}
\usepackage{multirow}
\usepackage{longtable}
\usepackage{tabularx}
\usepackage{diagbox}
\usepackage{lscape}

\usepackage[shortlabels,inline]{enumitem}

\usepackage{algorithm}
\usepackage{algorithmicx}
\usepackage{algpseudocode}

\usepackage{amsmath,amssymb,amstext,mathtools,mathrsfs}
\usepackage{amsthm} 
\usepackage{soul}

\usepackage{tikz}
\usepackage{tikz-cd}
\usetikzlibrary{matrix,calc,arrows.meta,decorations.pathreplacing,backgrounds,decorations.markings}
\usepackage{extpfeil}
\usepackage{scalerel}[2016/12/29]
\definecolor{hqgreen}{RGB}{65,117,5}
\definecolor{hqred}{RGB}{208,2,27}
\definecolor{hqblue}{RGB}{20,80,200}
\definecolor{hqvert}{RGB}{144,19,254}
\definecolor{hqorange}{RGB}{224,110,15}

\usepackage{listings}

\usepackage{euscript}
\usepackage{ifthen}
\usepackage{wrapfig}
\usepackage{verbatim}
\usepackage{soul}
\usepackage[normalem]{ulem}
\usepackage[makeroom]{cancel}
\usepackage{blkarray}
\usepackage{hhline}
\usepackage[overload]{empheq}
\usepackage{fancyhdr}
\usepackage{appendix}
\usepackage{geometry}
\usepackage{cellspace}
\usepackage{scalerel,stackengine}
\usepackage{calc} 
\stackMath
\usepackage{xr}

\usepackage{pifont}

\usepackage{xurl} 
\usepackage[colorlinks=true, linkcolor=blue, citecolor=blue, urlcolor=blue]{hyperref}

\usepackage[backend=biber,style=alphabetic,url=true,doi=false,isbn=false]{biblatex}
\definecolor{codegreen}{rgb}{0,0.6,0}
\definecolor{codegray}{rgb}{0.5,0.5,0.5}
\definecolor{codepurple}{rgb}{0.58,0,0.82}
\definecolor{backcolour}{rgb}{0.95,0.95,0.92}

\lstdefinestyle{mystyle}{
  backgroundcolor=\color{backcolour},
  commentstyle=\color{codegreen},
  keywordstyle=\color{magenta},
  numberstyle=\tiny\color{codegray},
  stringstyle=\color{codepurple},
  basicstyle=\ttfamily\footnotesize,
  breaklines=true,
  captionpos=b,
  keepspaces=true,
  numbers=left,
  numbersep=5pt,
  showspaces=false,
  showstringspaces=false,
  showtabs=false,
  tabsize=2
}
\numberwithin{equation}{section}

\newsavebox{\overlongequation}

\newcommand\reallywidehat[1]{%
  \savestack{\tmpbox}{\stretchto{%
    \scaleto{%
      \scalerel*[\widthof{\ensuremath{#1}}]{\kern-.6pt\bigwedge\kern-.6pt}%
      {\rule[-\textheight/2]{1ex}{\textheight}}%
    }{\textheight}%
  }{0.5ex}}%
  \stackon[1pt]{#1}{\tmpbox}%
}

\newcommand{\xdashrightarrow}[2][]{\ext@arrow 0359\rightarrowfill@@{#1}{#2}}

\usepackage{etoolbox} 

\newtoggle{mathdiamond}
\togglefalse{mathdiamond}

\newcommand{\eqdiamond}{%
  \tag*{$\lozenge$}
  \global\toggletrue{mathdiamond}
}

\newcommand{\envdiamond}{%
  \iftoggle{mathdiamond}{%
    \global\togglefalse{mathdiamond}
  }{%
    \unskip\nobreak\hfill\ensuremath{\lozenge}
  }%
}

\newtheorem{theorem}{Theorem}[section]
\newtheorem{lemma}[theorem]{Lemma}
\newtheorem{proposition}[theorem]{Proposition}

\newtheorem*{theorem-star}{Theorem}
\theoremstyle{definition}
\newtheorem{definition}[theorem]{Definition}

\newtheorem{remark}[theorem]{Remark}

\newtheorem{example}{Example}

\AtEndEnvironment{theorem}{\envdiamond}
\AtEndEnvironment{lemma}{\envdiamond}
\AtEndEnvironment{proposition}{\envdiamond}
\AtEndEnvironment{corollary}{\envdiamond}
\AtEndEnvironment{conjecture}{\envdiamond}
\AtEndEnvironment{fact}{\envdiamond}
\AtEndEnvironment{assumption}{\envdiamond}
\AtEndEnvironment{definition}{\envdiamond}
\AtEndEnvironment{notation}{\envdiamond}
\AtEndEnvironment{problem}{\envdiamond}
\AtEndEnvironment{remark}{\envdiamond}
\AtEndEnvironment{note}{\envdiamond}
\AtEndEnvironment{question}{\envdiamond}
\AtEndEnvironment{exercise}{\envdiamond}
\AtEndEnvironment{example}{\envdiamond}

\graphicspath{{./figures/}}

\newcommand{\sslash}{\mathbin{/\mkern-6mu/}}

\title[Harder-Narasimhan Filtrations of Hyperquiver Representations]{Harder-Narasimhan Filtrations of \\ Hyperquiver Representations}
\author{Tommi Muller}\address{Mathematical Institute, University of Oxford, OX2 6GG, UK}\email{tommi.muller@maths.ox.ac.uk} 
\date{14 July 2026}

\begin{document}

\begin{abstract}
\emph{Hyperquiver representations} are  generalisations of quiver representations. They are directed hypergraphs with vector spaces assigned to vertices, and for every directed hyperedge, an assignment of a linear map from the tensor product of the source vector spaces to the tensor product of the target vector spaces. We study the moduli of hyperquiver representations, where we use affine GIT to describe the (in)stability of such objects under the action of a product of general linear groups. Our main result introduces the \emph{Harder-Narasimhan filtration} of a hyperquiver representation. We also derive a formula for the weights of this unique filtration, which specialise to the Harder-Narasimhan slopes of a quiver representation.
\end{abstract}

\maketitle
\thispagestyle{empty}

\section{Introduction}

\subsection{Quivers}

Quiver representations are objects of great importance lying at the intersection of many areas including representation theory, combinatorics, algebra, and algebraic and symplectic geometry. They consist of a \emph{quiver} $\mathcal{Q}$, i.e., a graph with a set of vertices $\mathcal{V}$ and a set of directed edges $\mathcal{E}$ described by source and target vertex maps $s,t: \mathcal{E} \rightarrow \mathcal{V}$, together with a \emph{representation} $\boldsymbol{\phi}$ of $\mathcal{Q}$, i.e., an assignment of a vector space $W_v$ to each vertex $v \in \mathcal{V}$ and a linear map $\phi_e : W_{s(e)} \rightarrow W_{t(e)}$ on each edge $e \in \mathcal{E}$. The dimensions $d_v$ of the vector spaces $W_v$ are encoded in a \emph{dimension vector} $\mathbf{d}$.

Since Gabriel's Theorem classifying indecomposable representations of quivers~\cite{Gab72}, quivers and their representations have found significant applications. When the group $G = \prod_{v \in \mathcal{V}} \mathrm{GL}(W_v)$ acts on the space of representations of a quiver $\mathcal{Q}$ by change of basis over an algebraically closed field $k$, one may use geometric invariant theory (GIT)~\cite{MumfordGIT} to construct moduli spaces of such representations. These include, for example, partial flag varieties, finite-dimensional algebras~\cite{King94}, moduli of spectral sequences~\cite{Kapranov_Routis_2017}, and, via the McKay correspondence, resolutions of quotient singularities~\cite{McKay80}.

The specific moduli space obtained depends on an additional datum called a \emph{stability parameter} $\boldsymbol\theta$, which is given by a number $\theta_v \in \mathbb{Q}$ attached to each vertex $v \in \mathcal{V}$. Whether a representation with dimension vector $\mathbf{d}$ lies on the moduli space or its boundary depends on whether it is \emph{slope-(semi)stable}; every representation can be assigned a value
\begin{equation}{\label{eq:slope-quiver}}
    \mu_{\boldsymbol{\theta}}(\mathbf{d}) = -\frac{\sum_{v \in V} \theta_v \cdot d_v}{\sum_{v \in V} d_v}
\end{equation}
called its \emph{slope}, and a representation is slope-semistable if all of its subrepresentations have smaller slope, and is slope-stable if all of its proper subrepresentations have strictly smaller slope. The category $\mathrm{Rep}(\mathcal{Q})$ of finite-dimensional representations of a quiver $Q$ also admits the following property once $\boldsymbol{\theta}$ is chosen.

\begin{theorem-star}[{\cite[Proposition 2.5]{Reineke03}}]{\label{thm:quiver-hn-filtration}}
Every representation $\boldsymbol{\varphi}$ with dimension vector $\mathbf{d} \neq \boldsymbol{0}$ has a unique finite-length filtration by subrepresentations with dimension vectors
$$\boldsymbol{0} = \mathbf{d}^0 \leq \mathbf{d}^1 \leq \cdots \leq \mathbf{d}^{M-1} \leq \mathbf{d}^M = \mathbf{d}$$
where $M \geq 1$ such that
\begin{enumerate}
	\item every successive quotient representation with dimension vector $\mathbf{d}^{i} - \mathbf{d}^{i-1}$ is slope-semistable
	\item $\mu_{\boldsymbol{\theta}}(\mathbf{d}^{1} - \mathbf{d}^{0}) > \cdots > \mu_{\boldsymbol{\theta}}(\mathbf{d}^{M} - \mathbf{d}^{M-1})$
\end{enumerate}
\end{theorem-star}

This unique object associated to a quiver representation is known as the \emph{Harder-Narasimhan filtration}. It serves as an important source of invariants~\cite{fersztand2024harder} for quiver representations. The slope-semistable representations are precisely those whose Harder-Narasimhan filtration has length $M = 1$.

\subsection{Hyperquivers}

In this paper, we study a generalisation of quivers and their representations, known as hyperquiver representations, which were first introduced in \cite{MullerNandaSeigal26}. A hyperquiver $\mathcal{H}$ consists of a set of vertices $\mathcal{V}$ as before, but this time the set $\mathcal{E}$ consists of directed hyperedges, which can have multiple, possibly repeated, source and target vertices. These are described by source and target maps $s,t: \mathcal{E} \rightarrow \bigsqcup_{n \geq 1} \mathcal{V}^n$ giving the tuple of source and target vertices $s(e)$ and $t(e)$ for a (hyper)edge $e \in \mathcal{E}$. Note that in this paper, we allow (hyper)edges to have multiple target vertices, which was not considered in ~\cite{MullerNandaSeigal26} where only one target vertex is allowed. An example of a hyperquiver can be seen in Figure~\ref{fig:hyperquiver-first-example}, consisting of two vertices and two (hyper)edges, where the red and blue edge each has two source vertices and one target vertex.

\begin{figure}[htbp]
\centering
\tikzset{every picture/.style={line width=0.75pt}} 
\begin{tikzpicture}[x=0.75pt,y=0.75pt,yscale=-1,xscale=1]
\draw [color={rgb, 255:red, 208; green, 2; blue, 27 }  ,draw opacity=1 ][line width=1.5]    (389.82,367.12) .. controls (369,324.8) and (283.5,340.8) .. (279.5,390.3) ;
\draw [color={rgb, 255:red, 74; green, 144; blue, 226 }  ,draw opacity=1 ][line width=1.5]    (544,408.75) .. controls (590.04,361.23) and (522.87,310.28) .. (477.86,363.11) ;
\draw [shift={(476.5,364.75)}, rotate = 309.04] [color={rgb, 255:red, 74; green, 144; blue, 226 }  ,draw opacity=1 ][line width=1.5]    (14.21,-4.28) .. controls (9.04,-1.82) and (4.3,-0.39) .. (0,0) .. controls (4.3,0.39) and (9.04,1.82) .. (14.21,4.28)   ;
\draw [color={rgb, 255:red, 74; green, 144; blue, 226 }  ,draw opacity=1 ][line width=1.5]    (474.07,371.27) .. controls (455.5,390.75) and (493,461.75) .. (544,408.75) ;
\draw [color={rgb, 255:red, 208; green, 2; blue, 27 }  ,draw opacity=1 ][line width=1.5]    (389.82,367.12) .. controls (411.5,313.8) and (286.5,302.8) .. (280,385.8) ;
\draw [color={rgb, 255:red, 208; green, 2; blue, 27 }  ,draw opacity=1 ][line width=1.5]    (279.5,390.3) .. controls (275.54,449.7) and (356.85,443.43) .. (388.72,373.41) ;
\draw [shift={(389.67,371.27)}, rotate = 113.4] [color={rgb, 255:red, 208; green, 2; blue, 27 }  ,draw opacity=1 ][line width=1.5]    (14.21,-4.28) .. controls (9.04,-1.82) and (4.3,-0.39) .. (0,0) .. controls (4.3,0.39) and (9.04,1.82) .. (14.21,4.28)   ;
\draw [color={rgb, 255:red, 74; green, 144; blue, 226 }  ,draw opacity=1 ][line width=1.5]    (389.67,371.27) .. controls (386.17,430.52) and (490.5,464.75) .. (544,408.75) ;
\draw  [fill={rgb, 255:red, 144; green, 19; blue, 254 }  ,fill opacity=1 ] (389.97,362.98) .. controls (392.31,363.09) and (394.14,365.04) .. (394.06,367.33) .. controls (393.97,369.62) and (392.01,371.38) .. (389.67,371.27) .. controls (387.33,371.16) and (385.5,369.21) .. (385.58,366.92) .. controls (385.67,364.63) and (387.63,362.87) .. (389.97,362.98) -- cycle ;
\draw  [fill={rgb, 255:red, 144; green, 19; blue, 254 }  ,fill opacity=1 ] (474.37,362.98) .. controls (476.71,363.09) and (478.54,365.04) .. (478.46,367.33) .. controls (478.37,369.62) and (476.41,371.38) .. (474.07,371.27) .. controls (471.73,371.16) and (469.9,369.21) .. (469.98,366.92) .. controls (470.07,364.63) and (472.03,362.87) .. (474.37,362.98) -- cycle ;
\draw (360,359.3) node [anchor=north west][inner sep=0.75pt]    {};
\draw (490.4,360.3) node [anchor=north west][inner sep=0.75pt]    {};
\end{tikzpicture}
         \caption{A hyperquiver with two vertices and two hyperedges.}
         \label{fig:hyperquiver-first-example}
\end{figure}
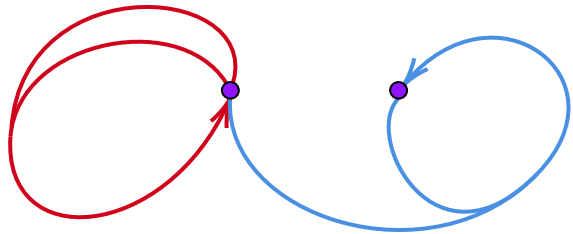

A representation $\boldsymbol{\phi}$ of a hyperquiver $\mathcal{H}$ is then defined to be an assignment of vector spaces $W_v$ to each vertex $v \in \mathcal{V}$ and a linear map $\phi_e$ from the tensor product of the source vector spaces $W_{s(e)} = W_{s(e)_1} \otimes \cdots \otimes W_{s(e)_m}$ to the tensor product of the target vector spaces $W_{t(e)} = W_{t(e)_1} \otimes \cdots \otimes W_{t(e)_n}$. A quiver with a representation is therefore the special case where every edge has exactly one source and target vertex. An algebra is an example of a hyperquiver representation with one vertex $v$ and a single hyperedge $(v,v) \rightarrow v$, and more generally, hyperquiver representations can encode many possible algebraic structures.

One can follow the same approach as before of using GIT to construct moduli spaces of hyperquiver representations by considering the space of representations $R(\mathcal{H},\mathbf{d}) = \bigoplus_{e \in \mathcal{E}} \mathrm{Hom}(W_{s(e)},W_{t(e)})$ under the action of $G = \prod_{v \in \mathcal{V}} \mathrm{GL}(W_v)$ as above. This leads one to ask whether representations parametrised by the moduli spaces can also be studied by some notion of slope and Harder-Narasimhan filtration. There are several challenges to this, which we will see and address in this paper:

\begin{enumerate}
    \item The category $\mathrm{Rep}(\mathcal{H})$ of representations of a hyperquiver $\mathcal{H}$ is not an abelian category unlike $\mathrm{Rep}(\mathcal{Q})$ for a quiver $\mathcal{Q}$, see Example~\ref{ex:finite-dimensional-algebras}. Hence, quotients by subobjects are not defined.
    \item Subobjects in $\mathrm{Rep}(\mathcal{H})$ are more difficult to describe than they are for $\mathrm{Rep}(\mathcal{Q})$.
    \item A notion of Harder-Narasimhan filtration can be defined for hyperquiver representations, but unlike quiver representations, slope-semistability of hyperquiver representations cannot be checked on length-1 filtrations, i.e., ``subrepresentations'', alone, see Example \ref{ex:hyperquiver-3-step}.
    \item The Harder-Narasimhan filtration of a hyperquiver representation additionally depends on a collection of weights that are analogous to the slopes $\mu_{\boldsymbol{\theta}}(\mathbf{d}^{i} - \mathbf{d}^{i-1})$ appearing in part (2) of the above theorem.
    \item A formula can be given for these weights that is almost the same as the slope formula \eqref{eq:slope-quiver}, but which contains an additional ``error term'' depending on the combinatorics of the underlying hyperquiver.
\end{enumerate}

In this paper, the key approach taken  to overcome these challenges is not to directly define slope and subrepresentations of hyperquiver representations, but instead to find descriptions of objects naturally attached to hyperquiver representations from the viewpoint of affine GIT. This approach was studied by Zamora~\cite{zamora-hn-paper} and Hoskins~\cite{Hoskins2012Stratifications} for quiver representations and Jones~\cite{jones2024canonicalfiltrationsfinitedimensionalalgebras} for finite-dimensional associative algebras and Lie algebras. 

\subsection{Harder-Narasimhan filtrations}

To develop a notion of slope and Harder-Narasimhan filtration for hyperquiver representations, we will describe them in terms of weights and maximally destabilising objects in affine GIT, i.e., the setting in GIT where a reductive group $G$ acts linearly on a vector space $X$. Fix a $G$-conjugation invariant norm $\|\cdot\|$ on the set of one-parameter subgroups (1-PSs) of $G$ induced by a Weyl-invariant inner product on the space of rational 1-PSs of a maximal torus of $G$. We can define the \emph{Kempf functional} $K$ which takes a representation $\boldsymbol{\phi} \in R(\mathcal{H},\mathbf{d})$ and produces the value
\begin{equation}{\label{eq:kempf-intro}}
    K(\boldsymbol{\phi}) = \inf \frac{(\chi_{\boldsymbol{\theta}},\lambda)}{\|\lambda\|}
\end{equation}
where $\chi_{\boldsymbol\theta} : G \rightarrow k^*$ is the character corresponding to the stability parameters $\boldsymbol{\theta}$, $(\cdot,\cdot)$ is the canonical pairing of characters and co-characters, and the infimum is over all 1-PSs \textit{admissible} for $\boldsymbol{\phi}$, i.e., the 1-PSs $\lambda$ such that $\lim_{t \rightarrow 0} \lambda(t)\cdot \boldsymbol\phi$ exists. A representation $\boldsymbol{\phi}$ is said to be \emph{$\boldsymbol{\theta}$-semistable} if $K(\boldsymbol{\phi}) \geq 0$, otherwise it is said to be \emph{$\boldsymbol{\theta}$-unstable}. If a representation $\boldsymbol{\phi}$ is $\boldsymbol{\theta}$-unstable, then Kempf's Theorem guarantees the existence of a 1-PS $\lambda : k^* \rightarrow G$ which minimises the Kempf functional, called a \emph{maximally destabilising 1-PS}, which is unique up to conjugation in the parabolic subgroup $P(\lambda)$ of $G$ determined by $\lambda$~\cite{kempfInstability}. The limit point $\lim_{t \rightarrow 0} \lambda(t)\cdot \boldsymbol\phi$ of a representation $\boldsymbol{\phi}$ under a maximally destabilising 1-PS $\lambda$ is called the \emph{associated graded} point $\mathrm{gr}(\boldsymbol{\phi})$. While $\mathrm{gr}(\boldsymbol{\phi})$ is not semistable respect to the stability parameters $\boldsymbol\theta$, it is semistable with respect to a different set of \emph{shifted stability parameters} $\tilde{\boldsymbol\theta}$, which will be described later in the paper.

For a representation $\boldsymbol{\phi} \in R(\mathcal{Q},\mathbf{d})$ of a quiver $\mathcal{Q}$, $\boldsymbol\theta$-semistability of $\boldsymbol{\phi}$ is equivalent to slope-semistability of $\boldsymbol{\phi}$ with respect to $\boldsymbol{\theta}$, under the assumption that $\langle \boldsymbol{\theta}, \mathbf{d} \rangle = 0$. This is a result known as King's Criterion~\cite{King94}. The condition $\langle \boldsymbol{\theta}, \mathbf{d} \rangle = 0$ eliminates central 1-PSs that destabilise every representation, which we explore for hyperquivers in Section~\ref{sec:destabilising-stability-params}. Furthermore, any 1-PS, in particular a maximally destabilising 1-PS $\lambda$ for a $\boldsymbol{\theta}$-unstable $\boldsymbol\phi$, splits the vector spaces $W_v$ into weight spaces $W_v = \bigoplus_{k=1}^m W_{v,k}$ and splits the edge maps ${\phi}_e$ into components $\phi_e^{ij} : W_{v,i} \rightarrow W_{v,j}$. Admissible 1-PSs necessarily have $\phi_e^{ij} = 0$ for $i > j$, and the limit under a maximally destabilising 1-PS $\lim_{t \rightarrow 0} \lambda(t) \cdot \boldsymbol{\phi}$ sends $\phi_e^{ij}$ to the zero map for $i < j$. Thus, the associated graded point $\mathrm{gr}(\boldsymbol{\phi})$ can be seen as a representation of $\mathcal{Q}$ assigning $\bigoplus_{k=1}^m W_{v,k}$ to each vertex $v \in \mathcal{V}$ and the linear maps $\mathrm{gr}(\boldsymbol{\phi})_e = (\mathrm{gr}(\boldsymbol{\phi})_e^{ij})_{i,j=1}^m$ to each edge $e \in \mathcal{E}$ with $\mathrm{gr}(\boldsymbol{\phi})_e^{ij} = \phi_e^{ij}$ for $i = j$ and 0 otherwise. Alternatively, one may view 1-PSs as inducing filtrations $F^k W_v = U_{v,k}$ on each vector space $W_k$, where $0 = U_{v,0} \subseteq U_{v,1} \subseteq \cdots \subseteq U_{v,m} = W_v$; the associated graded $\mathrm{gr}(\boldsymbol{\phi})$ then assigns $\bigoplus_{k=1}^m U_{v,k}/U_{v,k-1}$ to each vertex and the analogous linear maps to each edge as described above.

The Kempf functional is invariant under rational rescalings of 1-PSs, allowing us to view \eqref{eq:kempf-intro} as an optimisation problem whose minimiser can be found using the necessary and sufficient Karush-Kuhn-Tucker (KKT) Conditions from optimisation theory. This allows us to translate rational 1-PSs into the following central object of this paper, called a \emph{$\mathbb{Q}$-filtration}, see Definition~\ref{def:Q-filtration}. A $\mathbb{Q}$-filtration for a hyperquiver $\mathcal{H}$ with dimension vector $\mathbf{d}$ is a pair $(F^\bullet W,\boldsymbol{\ell})$ given by a filtration $F^\bullet W = (F^\bullet W_v)_{v \in \mathcal{V}}$ of each vector space $W_v$ as described above, together with a decreasing sequence of rational weights $\boldsymbol{\ell} = (\ell_k)_{k=1}^m$ with $\ell_1 > \cdots > \ell_m$. Likewise, we can translate admissible 1-PSs into objects called \emph{admissible $\mathbb{Q}$-filtrations}, in Definition~\ref{def:admissible-Q-filtration}.

We briefly explain what the analogue of the associated graded point for a hyperquiver representation should be. Let $\boldsymbol\phi \in R(\mathcal{H},\mathbf{d})$ be a $\boldsymbol\theta$-unstable representation and $(F^\bullet W,\boldsymbol{\ell})$ be an admissible $\mathbb{Q}$-filtration corresponding to a maximally destabilising 1-PS for $\boldsymbol\phi$. Denote $[m] := \{1,\ldots,m\}$. We define the \emph{refined hyperquiver} $\mathcal{H}_{F}$ to be the hyperquiver with vertex set $\mathcal{V} \times [m]$ and an edge $c = (e,I,J)$ for every $e \in \mathcal{E}$, $I : s(e) \rightarrow [m]$, and $J : t(e) \rightarrow [m]$, having source vertices $((s,I(s)))_{s \in s(e)}$ and target vertices $((t,J(t)))_{t \in t(e)}$. The associated graded point $\mathrm{gr}(\boldsymbol{\phi})$ is then seen as a hyperquiver representation of $H_{\boldsymbol\phi}$ with dimension vector $(d_{v,k})_{v \in \mathcal{V},k\in[m]}$ where $d_{v,k} = \dim(F^{k}W_v/F^{k-1}W_v)$. We denote this dimension vector by $\mathbf{d}_{\boldsymbol{\phi}}$. A full description, including what its edge maps should be, is found in Definition~\ref{def:refined-hyperquiver}.

The KKT Conditions allow us to both translate the existence and uniquness of a maximally destabilising rational 1-PS into a notion of a Harder-Narasimhan filtration of a hyperquiver representation, as well as to prove that the weights of this 1-PS admit a formula that is the hyperquiver analogue of the slope of a quiver representation. This is the main result of the paper. We present a simplified version below.

\begin{theorem-star}[Theorem~\ref{thm:hn-filtration}]
Let $\mathcal{H} = (\mathcal{V},\mathcal{E})$ be a hyperquiver, $\boldsymbol{\theta} \in \mathbb{Q}^{|\mathcal{V}|}$ be stability parameters, and $\boldsymbol{\phi} \in R(\mathcal{H},\mathbf{d})$ be a $\boldsymbol{\theta}$-unstable representation. Then there
exists a unique admissible $\mathbb{Q}$-filtration
$(F^\bullet W,\boldsymbol{\ell})$, the \emph{Harder-Narasimhan filtration}, such that 
\begin{enumerate}
  \item[(a)] for some suitably shifted stability parameters $\widetilde{\boldsymbol\theta}$ for $\mathcal{H}_{F}$, the associated graded hyperquiver representation $\mathrm{gr}(\boldsymbol{\phi}) \in R(\mathcal{H}_{F},\mathbf{d}_{\boldsymbol{\phi}})$ is $\widetilde{\boldsymbol\theta}$-semistable;
  \item[(b)]  the weights $\ell_k$ are determined by the formula 
\begin{equation}{\label{eq:HN-weights-formula-intro}}
\boldsymbol{\ell} = \boldsymbol{p}/\|\boldsymbol{p}\|_D^2 \quad \text{ for } \quad \boldsymbol{p} = -D^{-1} \left( I - A^\top L^\dagger A D^{-1}  \right)\boldsymbol{\Theta}
\end{equation}
where $\boldsymbol\Theta$, $D$, $A$, $L^\dagger$ depend only on the hyperquiver representation $\mathrm{gr}(\boldsymbol\phi)$.
\end{enumerate}
\end{theorem-star}

When $\mathcal{H}$ is a quiver, our theorem recovers the Harder-Narasimhan filtration for quiver representations; moreover, the $A^\top L^\dagger A D^{-1}$ term in \eqref{eq:HN-weights-formula-intro} vanishes for quivers and this formula reduces to $\frac{\boldsymbol{\ell}}{\|\ell\|_D^2} = -D^{-1}\boldsymbol{\Theta}$, which is precisely the slope formula \eqref{eq:slope-quiver} and the slopes $\mu_{\boldsymbol{\theta}}(\mathbf{d}^{i} - \mathbf{d}^{i-1})$ seen in Reineke's theorem as above~\cite{Reineke03}, see Remark~\ref{rem:quivers-hn-weights}.

\subsection{Outline}

In Section~\ref{sec:affine-GIT}, we review the necessary background on affine GIT. In Section~\ref{sec:setup}, we introduce hyperquivers, their representations, and study their semistability. Lastly, in Section~\ref{sec:hn-filtration}, we study instability of hyperquiver representations and state and prove the main theorem concerning the Harder-Narasimhan filtration.

\subsection{Conventions and notation}

Throughout, we work over an algebraically closed field $k$. We denote scalars by unbolded lowercase $x$, vectors by bold lowercase $\mathbf{x}$, and matrices by unbolded uppercase $X$. We write $[m] := \{1,\ldots,m\}$. All vector spaces are finite-dimensional. For a group $G$, $X_*(G)$ denotes the set of one-parameter subgroups $\lambda : k^* \rightarrow G$ of $G$ and $X^*(G)$ denotes the group of characters $\chi : G \rightarrow k^*$ of $G$.

\subsection{Acknowledgments}

I would like to express my deepest gratitude to both my supervisors Frances Kirwan and Vidit Nanda. I thank them for their unending patience, support, and guidance throughout this project. I would also like to thank Chenjing Bu and Kevin McGerty for helpful discussions.

\section{Affine GIT}{\label{sec:affine-GIT}}

We begin by reviewing background from affine GIT. We work over an algebraically closed field $k$. Let $G$ be a  reductive group acting on a vector space $V$, which we may view as the affine variety $\mathrm{Spec}\,k[V]$, where $k[V] = \mathrm{Sym}(V^*)$ is its coordinate ring. Let $L$ be a \textit{$G$-linearised} line bundle over $V$. This is a line bundle with an action $G \times L \rightarrow L$ such that the projection $\pi : L \rightarrow V$ is $G$-equivariant and the action of $G$ on the fibres of $L$ is a linear isomorphism, i.e., for every $g \in G$ and $v \in V$, the map on fibres $L_v \rightarrow L_{g \cdot v}$ given by $\ell \mapsto g \cdot \ell$ is a linear isomorphism. We denote a $G$-linearisation of $L$ by $\mathscr{L}$. The $G$-linearised line bundles over $V$ up to isomorphism form an abelian group under the tensor product $\otimes$ called the \textit{equivariant Picard group} $\mathrm{Pic}^G(V)$. When $G$ is the trivial group, $\mathrm{Pic}^G(V) = \mathrm{Pic}(V)$ is the Picard group of $V$. Furthermore, $\mathrm{Pic}(V)$ is the trivial group as the only line bundles affine space admits is the trivial line bundle $L = V \times k$. One can classify all $G$-linearised line bundles for any action of $G$ on $V$ as follows.

\begin{proposition}[{\cite[Lemma 2.2, Proposition 2.3]{KnopKraftVust1989Picard}}]
Let $G$ be a linear algebraic group acting on a vector space $V$. Then
\begin{equation*}
    \mathrm{Pic}^G(V) \cong \mathrm{Hom}(G,k^*), 
\end{equation*}
the group of characters of $G$.
\end{proposition}

In other words, the only line bundle lying over $V$ is the trivial line bundle $L = V \times k$, and all $G$-linearisations of $L$ are through characters $\chi: G \rightarrow k^*$, where $G$ acts on $L$ via
$$g\cdot (v,\ell) = (g\cdot v, \chi(g)\cdot \ell), \quad v \in V, \ell \in L_v.$$

\begin{example}
If $G = \mathrm{GL}(V)$, then every character $\chi : \mathrm{GL}(V) \rightarrow k^*$ is of the form
$$\chi(g) = \mathrm{det}(g)^\theta, \quad \theta \in \mathbb{Z}.$$
Similarly, if $V_i$ are vector spaces for $i = 1,\ldots,n$ and $G = \prod_{i=1}^n \mathrm{GL}(V_i)$, 
then every character $\chi : G \rightarrow k^*$ is of the following form sending $\boldsymbol{g} = (g_1,\ldots,g_n) \in G$ to
$$\chi(\boldsymbol{g}) = \prod_{i=1}^n \mathrm{det}(g_i)^{\theta_i}$$
for some $\boldsymbol{\theta} = (\theta_i)_{i=1}^n \in \mathbb{Z}^n$. We call $\boldsymbol{\theta}$ a \textbf{stability parameter} for $G = \prod_{i=1}^n \mathrm{GL}(V_i)$.
\end{example}

The induced $G$-linearisation of $L^{\otimes m}$ for all  $m \geq 0$ is given by
\begin{equation*}
    g\cdot (v,\ell) = (g\cdot v, \chi(g)^m\cdot \ell), \quad \forall v \in V, \, \forall \ell \in L_v.
\end{equation*}
Define the \textit{ring of sections} of $L$ to be
\begin{equation*}
    R(V,L) := \bigoplus_{m \geq 0} H^0(V,L^{\otimes m})
\end{equation*}
where for $\sigma \in H^0(V,L^{\otimes i})$ and $\sigma' \in H^0(V,L^{\otimes j})$, the multiplication is given by
\begin{equation*}
    (\sigma \cdot \sigma')(v) := \sigma(v)\otimes\sigma'(v) \in (L^{\otimes(i+j)})_v \quad \forall v \in V.
\end{equation*}
Since $L$ is the trivial line bundle, we have $H^0(V,L^{\otimes m}) = k[V] = \mathrm{Sym}(V^*)$ for all $m \geq 0$. Hence, $R(V,L) \cong k[V][t]$ is a graded $\mathrm{Sym}(V^*)$-algebra in the variable $t$ lying in degree 1. There is an action of $G$ on the sections $\sigma \in H^0(V,L)$ of $L$, given by
\begin{equation*}
    (g \cdot \sigma)(v) := g \cdot \sigma(g^{-1} \cdot v), \quad \forall g \in G, \, \forall v \in V
\end{equation*}
which extends component-wise to $R(V,L)$. We say that $\sigma$ is \textit{$G$-invariant} if $g \cdot \sigma = \sigma$, and denote the submodule of $G$-invariant sections of $L$ by $H^0(V,L)^G$. As $G$ is reductive, $R(V,L)^G = \bigoplus_{m \geq 0} H^0(V,L^{\otimes m})^G$ is finitely generated.

We can describe these $G$-invariant sections concretely as follows. A section $f \in H^0(V,L^{\otimes m})$ is $G$-invariant if $(g \cdot f)(v) = g \cdot f(g^{-1} \cdot v) = \chi(g)^m f(g^{-1} \cdot v) = f(v)$, or in other words, $f(g \cdot v) = \chi(g)^m f(v)$ for all $v \in V$ and $g \in G$. The sections $H^0(V,L^{\otimes m})^G$ satisfying this relation are called \textit{$\chi^m$-semi-invariants}.

\begin{definition}
Let $G$ be a reductive group acting on a vector space $V$, and let $L$ be a $G$-linearised line bundle over $V$. A point $v \in V$ is \textbf{semistable} for this linearisation if there exists $\sigma \in H^0(V,L^{\otimes N})^G$ for some $N > 0$ such that $\sigma(v) \neq 0$, otherwise it is \textbf{unstable}.
\end{definition}

The set of semistable points for a given linearisation forms an open subset of $V$, which we denote by $V^{\mathrm{ss},\mathscr{L}}$, or $V^{\mathrm{ss}}$ if the linearisation $\mathscr{L}$ is clear from context. Note that it is possible for $V^{\mathrm{ss},\mathscr{L}}$ to be empty, depending on $\mathscr{L}$. We denote the set of unstable points by $V^{\mathscr{L},\mathrm{u}} := V \setminus V^{\mathscr{L},\mathrm{ss}}$ or $V^{\mathrm{u}}$. If $G = \prod_{i=1}^n \mathrm{GL}(V_i)$, then the linearisation $\mathscr{L}$ depends on a stability parameter $\boldsymbol{\theta} \in \mathbb{Z}^n$, and we will call a semistable point in this setting a \textbf{$\boldsymbol{\theta}$-semistable} point and denote the set of such points by $V^{\boldsymbol{\theta}\text{-}\mathrm{ss}}$.

The \textit{GIT quotient} of a vector space $V$ by the action of a reductive group $G$ with respect to a $G$-linearisation $\mathscr{L}$ is the scheme
	    \begin{equation*}
			V \sslash_\mathscr{L} G := \mathrm{Proj} \; R(V,L)^G.
		\end{equation*}
Note that $V \sslash_\mathscr{L} G$ is not necessarily projective over $k$, i.e., a closed subscheme of $\mathbb{P}^n$. By the Proj construction on a finitely generated ring, it is projective over the affine scheme $\mathrm{Spec} \, (R(V,L)^G)_0 = \mathrm{Spec} \, k[V]^G$, so it is projective over $k$ if $k[V]^G = k$. In general, $V \sslash_\mathscr{L} G$ is a quasi-projective scheme. If $\mathscr{L}$ is the trivial linearisation given by the trivial character $\chi = 1$, then the GIT quotient $V \sslash_\mathscr{L} G = \mathrm{Proj}\,k[V]^G[t] \cong \mathrm{Spec}\,k[V]^G$ recovers classical invariant theory. Furthermore, there is a \textit{good quotient map}
	\begin{equation*}
		V^{\mathrm{ss},\mathscr{L}} \rightarrow V \sslash_\mathscr{L} G
	\end{equation*}
so that points in $V \sslash_\mathscr{L} G$ parametrise $G$-orbits in $V^{\mathrm{ss},\mathscr{L}}$, up to the equivalence that two $G$-orbits are identified if their closures intersect~\cite[Chapter 8.1]{Dolgachev_2003}. Evidently, the group $\mathrm{Pic}^G(V)$ controls all the possible GIT quotients $V \sslash_\mathscr{L} G$ one can construct. 

One can ask when exactly a point $v \in V$ is semistable for a given linearisation. This is answered by the \textit{Hilbert-Mumford Criterion} by checking a condition among \textbf{one-parameter subgroups} (\textbf{1-PSs}) or \textbf{cocharacters} $\lambda : k^* \rightarrow G$ of $G$.

Let $X^*(G)$ be the group of characters of $G$ and let and $X_*(G)$ be the set of cocharacters of $G$. Recall that there is a canonical map $(\cdot,\cdot) : X^*(G) \times X_*(G) \rightarrow \mathbb{Z}$ defined by $(\chi,\lambda) = m$ where $m$ is the exponent of the composition $\chi(\lambda(t)) = t^m$ for all $t \in k^*$. Note that the map $(\cdot,\cdot)$ is additive in the first argument and is a pairing if $G$ is abelian. For $v \in V$ and $\lambda \in X_*(G)$, we say that $\lim_{t \rightarrow 0} \lambda(t)\cdot v$ \textit{exists} if the map $\alpha : \mathbb{A}^1\setminus\{0\} \rightarrow V$ sending $t \mapsto \lambda(t)\cdot v$ extends to a map $\widetilde{\alpha} : \mathbb{A}^1 \rightarrow V$. If it does exist, then such an extension $\widetilde{\alpha}$ is unique as $V$ is separated, and the limit is defined to be $\lim_{t \rightarrow 0} \lambda(t)\cdot v := \widetilde{\alpha}(0)$, in which case we say that $\lambda$ is an \textbf{admissible} 1-PS for $v$. We denote the set of non-trivial admissible 1-PSs of $\lambda$ for $v$ to be $C_v$.

\begin{theorem}[Affine Hilbert-Mumford Criterion {\cite[Proposition 2.5]{Hoskins2012Stratifications}}]
Let $G$ be a reductive group acting on a vector space $V$, and let $\chi: G \rightarrow k^*$ be a character inducing the $G$-linearisation $\mathscr{L}$. Then $v \in V^{\mathrm{ss},\mathscr{L}}$ if and only if
\begin{equation*}
    (\chi, \lambda) \geq 0, \quad \forall \, \lambda \in C_v. \eqdiamond
\end{equation*}
\end{theorem}

\begin{remark}{\label{rem:pairing-scaling-semistability}}
Since $(\cdot,\cdot)$ satisfies $(\chi^N,\lambda) = N(\chi,\lambda)$ for $N \geq 1$, the sign of $(\chi,\lambda)$ is unchanged when scaling $\chi$ by a positive integer. Therefore, if $\mathscr{L}$ is the linearisation defined by $\chi$ and $\mathscr{L}^N$ is the linearisation defined by $\chi^N$, then by the Hilbert-Mumford Criterion, $V^{\mathscr{L},\mathrm{ss}} = V^{\mathscr{L}^N,\mathrm{ss}}$. In particular, if $G = \prod_{i=1}^n \mathrm{GL}(V_i)$ and $\boldsymbol{\theta} \in \mathbb{Z}^n$ is a stability parameter, then $V^{\boldsymbol{\theta},\mathrm{ss}} = V^{N\boldsymbol{\theta},\mathrm{ss}}$ for any $N \geq 1$.
\end{remark}

\begin{remark}{\label{rem:G-invaraint-subvariety}}
Let $Y \subseteq V$ be a closed $G$-invariant subvariety. Then for any 1-PS $\lambda : k^* \rightarrow G$, we have $\lambda(t)\cdot y \in Y$ for all $y \in Y$ and $t \neq 0$ as $Y$ is $G$-invariant, and if the limit at $t = 0$ exists in $V$, then it exists in $Y$ as $Y$ is closed. Since the pairing $(\chi,\lambda)$ is independent of $Y$, we can therefore consider more generally the GIT problem of $G$ acting on $Y$ with the linearisation $\mathscr{L}$, and we have $Y^{\mathrm{ss},\mathscr{L}} = Y \cap V^{\mathrm{ss},\mathscr{L}}$.
\end{remark}

\section{Hyperquivers, representations, and semistability}{\label{sec:setup}}

In this section, we introduce hyperquivers and their representations, and study their semistability.

\begin{definition}
A \textbf{hyperquiver} is a directed hypergraph $\mathcal{H} = (\mathcal{V},\mathcal{E},s,t)$ consisting of a finite set of vertices $\mathcal{V}$, a finite set of directed hyperedges $\mathcal{E}$, together with source and target functions $s,t: \mathcal{E} \rightarrow \bigsqcup_{n \geq 1} \mathcal{V}^n$, where each hyperedge $e \in \mathcal{E}$ has a finite tuple of source vertices $s(e)$ and a finite tuple of target vertices $t(e)$.
\end{definition}

For brevity, we may write $\mathcal{H} = (\mathcal{V},\mathcal{E})$ and refer to hyperedges as \textit{edges}. When indexing a sum, product, or other operation over $s(e)$ and $t(e)$, we will write $s \in s(e)$ and $t \in t(e)$ to mean indexing $s$ over $s(e)$ and $t$ over $t(e)$ in the order they appear in the tuples.

\begin{definition}
Let $\mathcal{H} = (\mathcal{V},\mathcal{E},s,t)$ be a hyperquiver. A \textbf{representation} of $\mathcal{H}$ is an assignment of non-zero finite-dimensional vector spaces $W_v$ to each vertex $v \in \mathcal{V}$, and a linear map $\phi_e : \bigotimes_{s \in s(e)} W_{s} \rightarrow \bigotimes_{t \in t(e)} W_{t}$ to each edge $e \in \mathcal{E}$. The \textbf{dimension vector} of the representation is $\mathbf{d} = (d_v)_{v \in \mathcal{V}}$ where $d_v := \dim(W_v)$. The \textit{space of hyperquiver representations} is the vector space
\begin{equation*}
    R(\mathcal{H},\mathbf{d}) = \bigoplus_{e \in \mathcal{E}} \, \mathrm{Hom}\left(\bigotimes_{s \in s(e)} W_{s}, \bigotimes_{t \in t(e)} W_{t}\right).
    \eqdiamond
\end{equation*}
\end{definition}

We call a point $\boldsymbol{\phi} = (\phi_e)_{e \in \mathcal{E}}$ in $R(\mathcal{H},\mathbf{d})$ a \textit{representation} of $\mathcal{H}$, for brevity. The representations of $\mathcal{H}$ form a category $\mathrm{Rep}(\mathcal{H})$, whose objects are $\boldsymbol{\phi} \in R(\mathcal{H},\mathbf{d})$ for some $\mathbf{d}$, and whose morphisms are collections of linear maps $f_v : W_v \rightarrow W'_v$ for every $v \in \mathcal{V}$ such that the diagram
\begin{center}
\begin{tikzcd}
\displaystyle\bigotimes_{i=1}^m W_{s_i(e)}
    \arrow[dd, "\bigotimes_{i=1}^m f_{s_i(e)}"']
    \arrow[rr, "\phi_e"]
&&
\displaystyle\bigotimes_{j=1}^n W_{t_j(e)}
    \arrow[dd, "\bigotimes_{j=1}^n f_{t_j(e)}"]
\\
\\
\displaystyle\bigotimes_{i=1}^m W'_{s_i(e)}
    \arrow[rr, "\phi'_e"']
&&
\displaystyle\bigotimes_{j=1}^n W'_{t_j(e)}
\end{tikzcd}
\end{center}
commutes for every $e \in \mathcal{E}$. As we will see in the next example, $\mathrm{Rep}(\mathcal{H})$ is in general not a pre-additive category, hence it is not an abelian category.

\begin{example}[Finite-dimensional algebras]{\label{ex:finite-dimensional-algebras}}
A finite-dimensional algebra $A$ can be drawn as a hyperquiver representation, shown in Figure~\ref{fig:algebra-hyperquiver}, where the hyperquiver has one vertex $v$, one edge $e: (v,v) \rightarrow v$, and the representation assigns $A$ to $v$ and the multiplication map $A \otimes A \rightarrow A$ to $e$. By Remark~\ref{rem:G-invaraint-subvariety}, when additional closed conditions are imposed such as associativity or the Jacobi identity, for example, then we obtain associative algebras and Lie algebras.

To see that $\mathrm{Rep}(\mathcal{H})$ is not pre-additive for the hyperquiver $\mathcal{H}$ in Figure~\ref{fig:algebra-hyperquiver}, let $A = k = \mathbb{C}$ with $\otimes$ being ordinary multiplication and $f : A \rightarrow A$ be the identity morphism. Then for $x,y \in A$, $(\mathrm{id} + \mathrm{id})(xy) = 2xy$ while $(\mathrm{id} + \mathrm{id})(x)\cdot(\mathrm{id} + \mathrm{id})(y) = (2x)(2y) = 4xy$ are not equal. Thus, $\mathrm{Rep}(\mathcal{H})$ for an arbitrary hyperquiver $\mathcal{H}$ is in general not pre-additive and hence not abelian.

\end{example}

\begin{figure}[htbp]
     \centering
\tikzset{every picture/.style={line width=0.7pt}} 
\begin{tikzpicture}[x=0.75pt,y=0.75pt,yscale=-1,xscale=1]
\draw [color={rgb, 255:red, 208; green, 2; blue, 27 }  ,draw opacity=1 ][line width=1.5]    (487.23,505.56) .. controls (573.32,405.3) and (292.58,400.84) .. (385.79,503.99) ;
\draw [shift={(387.23,505.56)}, rotate = 227.06] [color={rgb, 255:red, 208; green, 2; blue, 27 }  ,draw opacity=1 ][line width=1.5]    (14.21,-4.28) .. controls (9.04,-1.82) and (4.3,-0.39) .. (0,0) .. controls (4.3,0.39) and (9.04,1.82) .. (14.21,4.28)   ;
\draw [color={rgb, 255:red, 208; green, 2; blue, 27 }  ,draw opacity=1 ][line width=1.5]    (387.23,505.56) .. controls (378.75,545.8) and (469.75,528.3) .. (487.23,505.56) ;
\draw [color={rgb, 255:red, 208; green, 2; blue, 27 }  ,draw opacity=1 ][line width=1.5]    (387.23,505.56) .. controls (417.75,472.3) and (450.25,536.3) .. (487.23,505.56) ;
\draw  [fill={rgb, 255:red, 144; green, 19; blue, 254 }  ,fill opacity=1 ] (383.43,503.89) .. controls (384.4,501.75) and (386.88,500.77) .. (388.98,501.69) .. controls (391.08,502.61) and (391.99,505.09) .. (391.02,507.23) .. controls (390.06,509.36) and (387.57,510.34) .. (385.48,509.42) .. controls (383.38,508.5) and (382.47,506.02) .. (383.43,503.89) -- cycle ;
\draw (360,507.91) node [anchor=north west][inner sep=0.75pt]    {$A$};
\end{tikzpicture}
     \caption{A hyperquiver representation encoding an algebra $A \otimes A \rightarrow A$.}
     \label{fig:algebra-hyperquiver}
\end{figure}
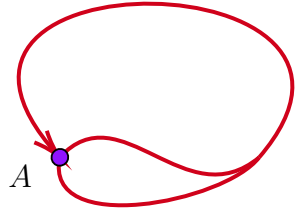

\begin{example}{\label{ex:lie-algebra-acting-on-rep}}
Let $U$ be a Lie group and $\rho: U \rightarrow \mathrm{GL}(V)$ be a representation. Then the differential $d\rho: \mathrm{Lie}(U) \rightarrow \mathfrak{gl}(V)$ gives rise to an action of $\mathrm{Lie}(U)$ on $V$ given by $X \cdot v := d\rho(X)(v)$, where $X \in \mathrm{Lie}(U)$ and $v \in V$. This scenario can be drawn as the hyperquiver representation in Figure~\ref{fig:lie-algebra-act-on-rep}, where the blue edge is the Lie algebra action $\mathrm{Lie}(U) \otimes V \rightarrow V$ and the red edge is the Lie bracket $[\cdot,\cdot]: \mathrm{Lie}(U) \otimes \mathrm{Lie}(U) \rightarrow \mathrm{Lie}(U)$.
\end{example}

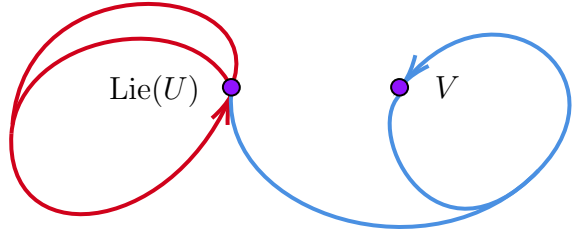
\begin{figure}[htbp]
\centering
\tikzset{every picture/.style={line width=0.75pt}} 
\begin{tikzpicture}[x=0.75pt,y=0.75pt,yscale=-1,xscale=1]
\draw [color={rgb, 255:red, 208; green, 2; blue, 27 }  ,draw opacity=1 ][line width=1.5]    (389.82,367.12) .. controls (369,324.8) and (283.5,340.8) .. (279.5,390.3) ;
\draw [color={rgb, 255:red, 74; green, 144; blue, 226 }  ,draw opacity=1 ][line width=1.5]    (544,408.75) .. controls (590.04,361.23) and (522.87,310.28) .. (477.86,363.11) ;
\draw [shift={(476.5,364.75)}, rotate = 309.04] [color={rgb, 255:red, 74; green, 144; blue, 226 }  ,draw opacity=1 ][line width=1.5]    (14.21,-4.28) .. controls (9.04,-1.82) and (4.3,-0.39) .. (0,0) .. controls (4.3,0.39) and (9.04,1.82) .. (14.21,4.28)   ;
\draw [color={rgb, 255:red, 74; green, 144; blue, 226 }  ,draw opacity=1 ][line width=1.5]    (474.07,371.27) .. controls (455.5,390.75) and (493,461.75) .. (544,408.75) ;
\draw [color={rgb, 255:red, 208; green, 2; blue, 27 }  ,draw opacity=1 ][line width=1.5]    (389.82,367.12) .. controls (411.5,313.8) and (286.5,302.8) .. (280,385.8) ;
\draw [color={rgb, 255:red, 208; green, 2; blue, 27 }  ,draw opacity=1 ][line width=1.5]    (279.5,390.3) .. controls (275.54,449.7) and (356.85,443.43) .. (388.72,373.41) ;
\draw [shift={(389.67,371.27)}, rotate = 113.4] [color={rgb, 255:red, 208; green, 2; blue, 27 }  ,draw opacity=1 ][line width=1.5]    (14.21,-4.28) .. controls (9.04,-1.82) and (4.3,-0.39) .. (0,0) .. controls (4.3,0.39) and (9.04,1.82) .. (14.21,4.28)   ;
\draw [color={rgb, 255:red, 74; green, 144; blue, 226 }  ,draw opacity=1 ][line width=1.5]    (389.67,371.27) .. controls (386.17,430.52) and (490.5,464.75) .. (544,408.75) ;
\draw  [fill={rgb, 255:red, 144; green, 19; blue, 254 }  ,fill opacity=1 ] (389.97,362.98) .. controls (392.31,363.09) and (394.14,365.04) .. (394.06,367.33) .. controls (393.97,369.62) and (392.01,371.38) .. (389.67,371.27) .. controls (387.33,371.16) and (385.5,369.21) .. (385.58,366.92) .. controls (385.67,364.63) and (387.63,362.87) .. (389.97,362.98) -- cycle ;
\draw  [fill={rgb, 255:red, 144; green, 19; blue, 254 }  ,fill opacity=1 ] (474.37,362.98) .. controls (476.71,363.09) and (478.54,365.04) .. (478.46,367.33) .. controls (478.37,369.62) and (476.41,371.38) .. (474.07,371.27) .. controls (471.73,371.16) and (469.9,369.21) .. (469.98,366.92) .. controls (470.07,364.63) and (472.03,362.87) .. (474.37,362.98) -- cycle ;
\draw (327,359.3) node [anchor=north west][inner sep=0.75pt]    {$\mathrm{Lie}(U)$};
\draw (490.4,360.3) node [anchor=north west][inner sep=0.75pt]    {$V$};
\end{tikzpicture}
         \caption{A hyperquiver representation corresponding to the action of a Lie algebra $\mathrm{Lie}(U)$ on a representation $V$ of its Lie group $U$.}
         \label{fig:lie-algebra-act-on-rep}
\end{figure}

\begin{example}{\label{ex:double-algebra}}
The following is an extension of the previous example and is a setting useful in non-reductive GIT~\cite{berczi-non-reductive}. Let $H = U \rtimes Q$ be the Levi decomposition of a linear algebraic group $H$, where $U$ is the unipotent radical of  $H$ and $Q$ is a reductive subgroup, acting by conjugation on $U$. Suppose $\lambda: k^* \rightarrow Z(Q)$ is a 1-PS such that the adjoint action of $\lambda(k^*)$ on $\text{Lie}(U)$ grades $\text{Lie}(U)$ into $Q$-module weight spaces $\text{Lie}(U)_i$ all with positive weights. Then $\lambda(k^*)$ decomposes $V$ into a direct sum of $Q$-module weight spaces $V_j$. From this, we obtain a hyperquiver representation shown in Figure~\ref{fig:doubly-graded-hyperquiver}.
\end{example}

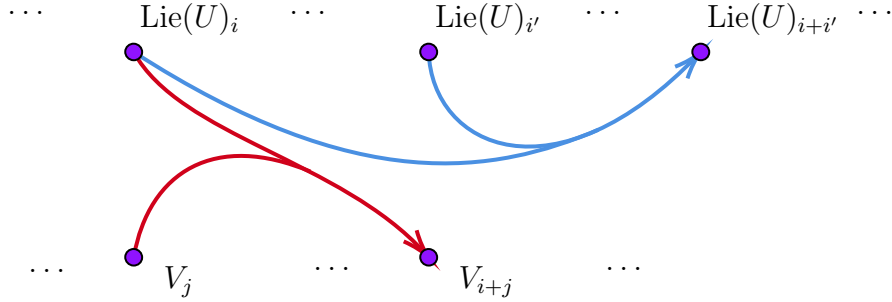
\begin{figure}[ht]

\tikzset{every picture/.style={line width=0.75pt}} 

\begin{tikzpicture}[x=0.75pt,y=0.75pt,yscale=-1,xscale=1]

\draw [color={rgb, 255:red, 74; green, 144; blue, 226 }  ,draw opacity=1 ][line width=1.5]    (328.33,523.79) .. controls (329.99,572.05) and (391.88,600.02) .. (463.75,524.93) ;
\draw [shift={(464.83,523.79)}, rotate = 133.39] [color={rgb, 255:red, 74; green, 144; blue, 226 }  ,draw opacity=1 ][line width=1.5]    (14.21,-4.28) .. controls (9.04,-1.82) and (4.3,-0.39) .. (0,0) .. controls (4.3,0.39) and (9.04,1.82) .. (14.21,4.28)   ;
\draw [color={rgb, 255:red, 74; green, 144; blue, 226 }  ,draw opacity=1 ][line width=1.5]    (180.33,523.79) .. controls (293.5,596.8) and (361,584.8) .. (413.5,562.8) ;
\draw [color={rgb, 255:red, 208; green, 2; blue, 27 }  ,draw opacity=1 ][line width=1.5]    (180.33,523.79) .. controls (203.65,566.15) and (284.04,577.94) .. (326.42,624.63) ;
\draw [shift={(328.33,626.79)}, rotate = 229.22] [color={rgb, 255:red, 208; green, 2; blue, 27 }  ,draw opacity=1 ][line width=1.5]    (14.21,-4.28) .. controls (9.04,-1.82) and (4.3,-0.39) .. (0,0) .. controls (4.3,0.39) and (9.04,1.82) .. (14.21,4.28)   ;
\draw [color={rgb, 255:red, 208; green, 2; blue, 27 }  ,draw opacity=1 ][line width=1.5]    (180.33,626.79) .. controls (188,583.3) and (221.5,563.8) .. (269,583.8) ;
\draw  [fill={rgb, 255:red, 144; green, 19; blue, 254 }  ,fill opacity=1 ] (180.48,622.64) .. controls (182.83,622.76) and (184.65,624.7) .. (184.57,626.99) .. controls (184.49,629.28) and (182.52,631.04) .. (180.18,630.93) .. controls (177.84,630.82) and (176.01,628.87) .. (176.1,626.58) .. controls (176.18,624.3) and (178.14,622.53) .. (180.48,622.64) -- cycle ;
\draw  [fill={rgb, 255:red, 144; green, 19; blue, 254 }  ,fill opacity=1 ] (180.48,519.64) .. controls (182.83,519.76) and (184.65,521.7) .. (184.57,523.99) .. controls (184.49,526.28) and (182.52,528.04) .. (180.18,527.93) .. controls (177.84,527.82) and (176.01,525.87) .. (176.1,523.58) .. controls (176.18,521.3) and (178.14,519.53) .. (180.48,519.64) -- cycle ;
\draw  [fill={rgb, 255:red, 144; green, 19; blue, 254 }  ,fill opacity=1 ] (328.48,622.64) .. controls (330.83,622.76) and (332.65,624.7) .. (332.57,626.99) .. controls (332.49,629.28) and (330.52,631.04) .. (328.18,630.93) .. controls (325.84,630.82) and (324.01,628.87) .. (324.1,626.58) .. controls (324.18,624.3) and (326.14,622.53) .. (328.48,622.64) -- cycle ;
\draw  [fill={rgb, 255:red, 144; green, 19; blue, 254 }  ,fill opacity=1 ] (328.48,519.64) .. controls (330.83,519.76) and (332.65,521.7) .. (332.57,523.99) .. controls (332.49,526.28) and (330.52,528.04) .. (328.18,527.93) .. controls (325.84,527.82) and (324.01,525.87) .. (324.1,523.58) .. controls (324.18,521.3) and (326.14,519.53) .. (328.48,519.64) -- cycle ;
\draw  [fill={rgb, 255:red, 144; green, 19; blue, 254 }  ,fill opacity=1 ] (464.98,519.64) .. controls (467.33,519.76) and (469.15,521.7) .. (469.07,523.99) .. controls (468.99,526.28) and (467.02,528.04) .. (464.68,527.93) .. controls (462.34,527.82) and (460.51,525.87) .. (460.6,523.58) .. controls (460.68,521.3) and (462.64,519.53) .. (464.98,519.64) -- cycle ;

\draw (126,630.91) node [anchor=north west][inner sep=0.75pt]    {$\dotsc $};
\draw (194,629.91) node [anchor=north west][inner sep=0.75pt]    {$V_{j}$};
\draw (269,629.91) node [anchor=north west][inner sep=0.75pt]    {$\dotsc $};
\draw (115.5,501.91) node [anchor=north west][inner sep=0.75pt]    {$\dotsc $};
\draw (182,497.91) node [anchor=north west][inner sep=0.75pt]    {$\text{Lie}( U)_{i}$};
\draw (257,501.91) node [anchor=north west][inner sep=0.75pt]    {$\dotsc $};
\draw (342,629.91) node [anchor=north west][inner sep=0.75pt]    {$V_{i+j}$};
\draw (330,497.91) node [anchor=north west][inner sep=0.75pt]    {$\text{Lie}( U)_{i'}$};
\draw (405,501.91) node [anchor=north west][inner sep=0.75pt]    {$\dotsc $};
\draw (466.5,497.91) node [anchor=north west][inner sep=0.75pt]    {$\text{Lie}( U)_{i+i'}$};
\draw (541.5,501.91) node [anchor=north west][inner sep=0.75pt]    {$\dotsc $};
\draw (415.5,629.91) node [anchor=north west][inner sep=0.75pt]    {$\dotsc $};

\end{tikzpicture}
         \caption{A doubly-graded version of Figure~\ref{fig:lie-algebra-act-on-rep} where $\mathrm{Lie}(U)$ is internally graded by a 1-PS $\lambda$ and grades $V$.}
         \label{fig:doubly-graded-hyperquiver}
\end{figure}

\subsection{Semistability of hyperquiver representations}{\label{sec:semistability-of-hyperquivers}}

Our affine GIT setting in this paper involves \textit{linear} actions of a reductive group $G$ on a vector space $V$, i.e., an action such that there exists a representation $\rho: G \rightarrow \mathrm{GL}(V)$ such that $g \cdot v = \rho(g)v$ for all $v \in V$. We will consider the group
$$G = \prod_{v \in \mathcal{V}} \mathrm{GL}(W_v)$$
acting on $V = R(\mathcal{H},\mathbf{d})$ via change of basis. Explicitly, if $\boldsymbol{\phi} = (\phi_e)_{e \in \mathcal{E}} \in R(\mathcal{H},\mathbf{d})$ and $\boldsymbol{g} = (g_v)_{v \in \mathcal{V}} \in G$, then
\begin{equation}{\label{eq:group-action-hyperquiver}}
    (\boldsymbol{g} \cdot \boldsymbol{\phi})_{e} = \left(\bigotimes_{t \in t(e)} g_{t}\right) \circ \phi_{e} \circ \left(\bigotimes_{s \in s(e)}(g_{s})^{-1}\right), \quad \forall e \in \mathcal{E}
\end{equation}
This action is linear as \eqref{eq:group-action-hyperquiver} is linear in $\phi_e$ and respects the product in $G$ by the \textit{mixed-product property} of tensors.

Let $\lambda : k^* \rightarrow G$ be a 1-PS. Then each component $\lambda_v : k^* \rightarrow \mathrm{GL}(W_v)$ of $\lambda$ is a (rational) representation of $k^*$, i.e., a representation that is also a homomorphism of linear algebraic groups. Every rational representation of $k^*$ splits into the direct sum of 1-dimensional representations where $\lambda(t)$ acts by $t^i$ for some $i \in \mathbb{Z}$. We therefore have a weight space decomposition
$$W_v = \bigoplus_{i \in I(\lambda)} W_v^{(i)}, \quad W_v^{(i)} = \{w \in W_v \;|\; \lambda(t) w = t^iw, \;\forall t\in k^*\}, \quad \forall v \in \mathcal{V}$$
where $I(\lambda) \subset \mathbb{Z}$ is the set of weights by which $\lambda$ acts. This can also be seen as a filtration $F^\bullet W = (F^\bullet W_v)_{v \in \mathcal{V}}$ of each vector space $W_v$ where $F^{j}W_v = \bigoplus_{i \in I(\lambda), i \geq j} W_v^{(i)}$. The subgroup $P(F^\bullet W)$ of $G$ preserving such a filtration is
\begin{equation}{\label{eq:parabolic-for-1-PS-intro}}
    P(F^\bullet W) = \{\boldsymbol{g} = (g_v)_{v \in V} \;|\; g_v(F^jW_v) \subseteq F^jW_v, \; \forall j \in \mathbb{Z}, \; \forall v \in \mathcal{V}\}
\end{equation}
A 1-PS $\lambda \in X_*(G)$ therefore uniquely determines a filtration $F^\bullet W$ together with its set of weights $I(\lambda)$, and conversely, the filtration $F^\bullet W$ and the weights $I(\lambda)$ uniquely determine $\lambda$ up to conjugation by $P(F^\bullet W)$. Additionally, for $e \in \mathcal{E}$, we have
\begin{align*}
    \mathrm{Hom}\left(\bigotimes_{s \in s(e)} W_{s}, \bigotimes_{t \in t(e)} W_{t}\right) &= \mathrm{Hom}\left(\bigotimes_{s \in s(e)} \bigoplus_{i \in \mathbb{Z}} W_{s}^{(i)}, \bigotimes_{t \in t(e)} \bigoplus_{j \in \mathbb{Z}} W_{t}^{(j)}\right) \\
    &= \bigoplus_{\substack{\mathbf{m} \in \mathbb{Z}^{|s(e)|} \\ \mathbf{n} \in \mathbb{Z}^{|t(e)|}}} \mathrm{Hom}\left(\bigotimes_{s \in s(e)} W_{s}^{(m_s)}, \bigotimes_{t \in t(e)} W_{t}^{(n_t)}\right)
\end{align*}
so by \eqref{eq:group-action-hyperquiver}, $\lambda(t)$ acts on the $(\mathbf{m},\mathbf{n})$-th component of $\phi_e$ by multiplying it by $t^{N(\mathbf{m},\mathbf{n})}$, where $N(\mathbf{m},\mathbf{n}) = \sum_{t \in t(e)} n_t - \sum_{s \in s(e)} m_s$. In particular, $\lim_{t \rightarrow 0} \lambda(t) \cdot \boldsymbol\phi$ exists if for all $(\mathbf{m},\mathbf{n})$-th components of $\phi_e$ which are non-zero, we have $N(\mathbf{m},\mathbf{n}) \geq 0$. We generalise these ideas into the following main objects of this paper.

\begin{definition}{\label{def:Q-filtration}}
Let $\mathcal{H} = (\mathcal{V},\mathcal{E})$ be a hyperquiver and $\mathbf{d} = (d_v)_{v \in \mathcal{V}}$ be a dimension vector. A \textbf{$\mathbb{Q}$-filtration} $(F^\bullet W,\boldsymbol{\ell})$ is
\begin{enumerate}
	\item a decreasing sequence of rational weights $\boldsymbol{\ell} = (\ell_k)_{k=1}^m$ with $\ell_1 > \cdots > \ell_m$, for some $m \geq 1$
	\item a filtration $F^\bullet W = (F^\bullet W_v)_{v \in \mathcal{V}}$ of each vector space $W_v$,
	$$\{0\} = U_{v,0} \subseteq U_{v,1} \subseteq \dots \subseteq U_{v,m} = W_v$$
	where $F^k W_v = U_{v,k}$, such that for all $k = 1,\ldots,m$, there is some $v \in \mathcal{V}$ with $U_{v,k-1} \subsetneq U_{v,k}$
\end{enumerate}    
\end{definition}

A $\mathbb{Q}$-filtration $(F^\bullet W,\boldsymbol{\ell})$ is said to be \textbf{trivial} if $m = 1$, $\ell_1 = 0$, and $\{0\} = F^0W_v \subseteq F^1W_v = W_v$ for all $v \in \mathcal{V}$.

\begin{definition}{\label{def:admissible-Q-filtration}}
A non-trivial $\mathbb{Q}$-filtration $(F^\bullet W,\boldsymbol{\ell})$ is said to be \textbf{admissible} for a representation $\boldsymbol{\phi} \in R(\mathcal{H},\mathbf{d})$ if for any edge $e \in \mathcal{E}$ and any choice of weights in the source $I : s(e) \rightarrow \{1,\ldots,m\}$, we have
\begin{equation}{\label{eq:admissible-filtration-property}}
\phi_e\left(\bigotimes_{s \in s(e)} U_{s, I(s)}\right) \subseteq \; \sum_{\substack{J : t(e) \rightarrow \{1,\ldots,m\} \\ \sum_{t \in t(e)} \ell_{J(t)} \ge \sum_{s \in s(e)} \ell_{I(s)}}} \; \bigotimes_{t \in t(e)} U_{t, J(t)}
\end{equation}
where the sum is over all choices of weights in the target $J : t(e) \rightarrow \{1,\ldots,m\}$.
\end{definition}

\begin{definition}
Let $\sim$ be the equivalence relation on the set $X_*(G) \times \mathbb{Q}_{> 0}$ where $(\lambda,q) \sim (\mu,r)$ if and only if there are $m,n \geq 1$ such that $\lambda^m = \mu^n$ and $\frac{q}{m} = \frac{r}{n}$. We define $X_*(G)_{\mathbb{Q}} := (X_*(G) \times \mathbb{Q}_{> 0})/\sim$ and call an element $[\lambda,q] \in X_*(G)_{\mathbb{Q}}$ a \textbf{rational 1-PS} of $G$.
\end{definition}

If $T$ is a maximal torus of $G$, then $X_*(T)$ is a free $\mathbb{Z}$-module of rank $\dim(T)$. Furthermore, the rational 1-PSs of $T$, as in the above definition, are in bijection with $X_*(T) \otimes_\mathbb{Z} \mathbb{Q}$.

\begin{remark}
A $\mathbb{Q}$-filtration of $\boldsymbol{\phi}$ corresponds to a rational 1-PS $[\lambda,q] \in X_*(G)_{\mathbb{Q}}$ of $G$, up to conjugation of $\lambda$ by $P(F^\bullet W)$, defined in \eqref{eq:parabolic-for-1-PS-intro}. A $\mathbb{Q}$-filtration admissible for $\boldsymbol{\phi}$ corresponds to a rational 1-PS $[\lambda,q] \in X_*(G)_{\mathbb{Q}}$ of $G$, up to conjugation of $\lambda$ by $P(F^\bullet W)$, which is admissible for $\boldsymbol{\phi}$, i.e., $\lim_{t \rightarrow 0} \lambda(t) \cdot \boldsymbol{\phi}$ exists.
\end{remark}

We will also consider a \textbf{rational character} $\chi_{\boldsymbol\theta} \in X^*(G)_{\mathbb{Q}}$ parametrised by a rational stability parameter $\boldsymbol{\theta} \in \mathbb{Q}^{|\mathcal{V}|}$, where $X^*(G)_{\mathbb{Q}} := X^*(G) \otimes \mathbb{Q}$. The canonical map $(\cdot,\cdot) : X^*(G) \times X_*(G) \rightarrow \mathbb{Z}$ can be extended rationally to $(\cdot,\cdot) : X^*(G)_{\mathbb{Q}} \times X_*(G)_{\mathbb{Q}} \rightarrow \mathbb{Q}$, where $(\chi \otimes q,[\lambda,r]) := qr(\chi,\lambda)$ and extending by linearity in the first argument. As $r$ is positive, this does not change the sign of $(\chi_{\boldsymbol{\theta}},\lambda)$. Furthermore, by Remark~\ref{rem:pairing-scaling-semistability}, we can extend the set of semistable points to rational characters $\chi_{\boldsymbol{\theta}}$ by defining $V^{\boldsymbol\theta\text{-}\mathrm{ss}} := V^{N\boldsymbol\theta\text{-}\mathrm{ss}}$ for any $N \geq 1$ making $N\boldsymbol{\theta}$ an integer vector.

The Hilbert-Mumford Criterion can therefore be extended to rational (co)characters. Hence, $\boldsymbol{\phi}$ is $\boldsymbol\theta$-semistable if and only if for all admissible $\mathbb{Q}$-filtrations $(F^\bullet \boldsymbol{\phi},\boldsymbol{\ell})$, we have
\begin{equation}{\label{eq:hilbert-mumford-filtration}}
    (\chi_{\boldsymbol{\theta}}, \lambda) = \sum_{v \in \mathcal{V}} \theta_v \sum_{k=1}^m \ell_k \dim(F^kW_v / F^{k-1}W_v) = \sum_{v \in \mathcal{V}} \theta_v w_v \geq 0
\end{equation}
where $w_v := \sum_{k=1}^m \ell_k \dim(F^kW_v / F^{k-1}W_v)$.

\subsection{Destabilising stability parameters}{\label{sec:destabilising-stability-params}}

In \cite{BehrendNoohi}, Behrend and Noohi studied moduli of graded algebras, which are essentially the same as the representation in Figure~\ref{fig:doubly-graded-hyperquiver} but with only the top row of vertices and blue edges. They imposed some additional conditions on their stability parameters $\boldsymbol\theta \in \mathbb{Z}^{|\mathcal{V}|}$ to ensure that the semistable locus was not empty for trivial reasons. This is due to the presence of central 1-PSs which destabilise every representation. In this paper, we do not impose any constraints on the stability parameters $\boldsymbol\theta$, nor are we concerned with whether the semistable locus is empty or not. However, in this section, we will investigate general conditions, applicable to all hyperquivers, for choosing stability parameters $\boldsymbol\theta$ which do not automatically make every point unstable. In doing so, we also generalise the semitability criterion \cite[Proposition 3.4]{BehrendNoohi} to more general hyperquiver representations.

Let $\mathcal{H}$ be a hyperquiver with no isolated vertices, i.e., every vertex is incident to some edge. Define an $|\mathcal{E}| \times |\mathcal{V}|$ matrix $M$ whose entries are
\begin{equation}{\label{eq:incidence-matrix}}
    M_{e,v} = \mathrm{mult}_{t(e)}(v) - \mathrm{mult}_{s(e)}(v), \quad \forall v \in \mathcal{V}, \; \forall e \in \mathcal{E}
\end{equation}
where $\mathrm{mult}_{t(e)}(v)$ and $\mathrm{mult}_{s(e)}(v)$ are the number of times $v$ appears in the tuples $t(e)$ and $s(e)$, respectively. Let $\lambda : k^* \rightarrow G$ be a 1-PS. If $\lambda$ acts trivially on all of $R(\mathcal{H},\mathbf{d})$, then it must be central, i.e.\ commute with all elements of $G$ for all $t \in k^*$. A central 1-PS $\lambda$ can be expressed as a tuple of integers $\boldsymbol{\gamma} = (\gamma_v)_{v \in \mathcal{V}} \in \mathbb{Z}^{|\mathcal{V}|}$, via
$$\lambda(t)_v = t^{\gamma_v}\cdot\mathrm{Id}_{W_v}, \quad \forall v \in \mathcal{V}$$
It acts trivially on $R(\mathcal{H},\mathbf{d})$ if and only if
$$\sum_{t \in t(e)} \gamma_t - \sum_{s \in s(e)} \gamma_s = 0, \quad \forall e \in \mathcal{E}$$
Therefore, the central 1-PSs of $G$ which act trivially on $R(\mathcal{H},\mathbf{d})$ form a central torus $T_{\mathrm{triv}}$ whose cocharacter lattice is $\mathrm{ker}(M)$. We thus consider 1-PSs of
$$\widetilde{G} = G / T_{\mathrm{triv}}$$
Let
\begin{align*}
    \chi_\theta: G &\longrightarrow k^* \\
    (g_v)_{v \in \mathcal{V}} &\mapsto \prod_{v \in \mathcal{V}}^{} \text{det}(g_v)^{\theta_v}
\end{align*}
be a character determined by $(\theta_v)_{v \in \mathcal{V}} \in \mathbb{Z}^{|\mathcal{V}|}$. Then for $\chi_\theta$ to descend to a character $\widetilde{\chi_\theta} : \widetilde{G} \rightarrow k^*$, we require that
\begin{equation}{\label{eq:character-descent}}
    \sum_{v \in \mathcal{V}} \theta_v d_v \gamma_v = 0, \quad \forall \boldsymbol{\gamma} \in \mathrm{ker}(M)    
\end{equation}
In fact, \eqref{eq:hilbert-mumford-filtration} is invariant under addition of central 1-PSs $\boldsymbol{\gamma} \in \mathrm{ker}(M)$. This is because if $\lambda'$ is $\lambda$ shifted by $\boldsymbol{\gamma}$, then the corresponding $w_v'$ is$$w_v' = \sum_{n \in \mathbb{Z}} (n + \gamma_v)\dim(W_v^{(n)}) = w_v + d_v\gamma_v$$and so\begin{align*}    (\chi_\theta, \lambda') = \sum_{v \in \mathcal{V}} \theta_v(w_v + d_v\gamma_v) = \sum_{v \in \mathcal{V}}\theta_vw_v + \underbrace{\sum_{v \in \mathcal{V}}\theta_vd_v\gamma_v}_{= 0} = \sum_{v \in \mathcal{V}}\theta_vw_v\end{align*}by \eqref{eq:character-descent}. Thus, $\boldsymbol{\phi}$ is semistable under the action of $\widetilde{G}$ if and only if the same above conditions hold.

The condition \eqref{eq:character-descent} is a general form of condition (2) of \cite[Definition 3.2]{BehrendNoohi}
Let $\mathbf{a} = (a_v)_{v \in \mathcal{V}}$ where $a_v := \theta_v d_v$ for all $v \in \mathcal{V}$. Then \eqref{eq:character-descent} says that $\mathbf{a} \in \mathrm{ker}(M)^{\perp} = \mathrm{Im}(M^T)$. Thus, there exists an $\boldsymbol{\eta} = (\eta_e)_{e \in \mathcal{E}}$ such that $\mathbf{a} = M^T\boldsymbol{\eta}$, i.e.
\begin{equation}{\label{eq:a_v-eta_e}}
    a_v = \sum_{e \in \mathcal{E}} \eta_e M_{e,v}, \quad \forall v \in \mathcal{V}
\end{equation}

More generally, a central 1-PS parameterised by $\boldsymbol{\gamma} \in \mathbb{Z}^{|\mathcal{V}|}$ will have the property that $\lim_{t \rightarrow 0} \lambda(t) \cdot \boldsymbol{\phi}$ exists for all $\boldsymbol{\phi} \in R(\mathcal{H},\mathbf{d})$ if and only if it acts on $R(\mathcal{H},\mathbf{d})$ with non-negative weights on every hyperedge, i.e.\ $M\boldsymbol{\gamma} \geq 0$ component-wise. If we choose stability parameters $(\theta_v)_{v \in \mathcal{V}}$ such that $\langle \boldsymbol{}\theta, \boldsymbol{\gamma}\rangle < 0$ for such an admissible central 1-PS, then every point of $R(\mathcal{H},\mathbf{d})$ would be unstable. Thus, to prevent the semistable locus from being empty for trivial reasons, we require the parameter feasibility condition:
\begin{equation}{\label{eq:farkas-lemma-stability}}
    \mathbf{a}^T\boldsymbol{\gamma} \geq 0,\text{ for all }\boldsymbol{\gamma}\text{ s.t. }M\boldsymbol{\gamma} \geq 0
\end{equation}
Notice that \eqref{eq:character-descent} is a special case of \eqref{eq:farkas-lemma-stability}. By Farkas' Lemma, $\mathbf{a}$ satisfies \eqref{eq:farkas-lemma-stability} if and only if there exists an $\boldsymbol{\eta} \geq 0$ such that $\mathbf{a} = M^T\boldsymbol{\eta}$. Thus, we can always choose $\eta_e \geq 0$ for all $e \in \mathcal{E}$ in \eqref{eq:a_v-eta_e}, and as $\mathbf{a}$ and $M$ are integer-valued, we can take a rational such $\boldsymbol{\eta}$.

Now writing $\theta_v = \frac{a_v}{d_v}$ for $v \in \mathcal{V}$ and substituting \eqref{eq:a_v-eta_e} and the definition of $M$ \eqref{eq:incidence-matrix}, we see that
\begin{align*}
    \sum_{v \in \mathcal{V}} \theta_v w_v = \sum_{v \in \mathcal{V}} a_v \frac{w_v}{d_v} &= \sum_{v \in V} \left(\sum_{e \in \mathcal{E}} \eta_e M_{e,v}\right) \frac{w_v}{d_v}
    = \sum_{e \in \mathcal{E}} \eta_e \left(\sum_{v \in \mathcal{V}} M_{e,v} \frac{w_v}{d_v}\right) \\
    &= \sum_{e \in \mathcal{E}} \eta_e \left(\sum_{t \in t(e)} \frac{w_t}{d_t} - \sum_{s \in s(e)} \frac{w_s}{d_s}\right).
\end{align*}
Therefore, the semistabilty condition \eqref{eq:hilbert-mumford-filtration} becomes
\begin{align}
    \sum_{e \in \mathcal{E}} \eta_e \left(\sum_{t \in t(e)} \frac{w_t}{d_t} - \sum_{s \in s(e)} \frac{w_s}{d_s}\right) \geq 0 {\label{eq:slope-stability-condition-eta}}
\end{align}
Thus, we have a duality: one can assign stability parameters $(\theta_v)_{v \in \mathcal{V}}$ to every vertex and check semistability using the Hilbert-Mumford criterion \eqref{eq:hilbert-mumford-filtration}, or one can assign non-negative stability parameters $(\eta_e)_{e \in \mathcal{E}}$ to every hyperedge and check semistability using the condition \eqref{eq:slope-stability-condition-eta}. The conditions \eqref{eq:farkas-lemma-stability} and \eqref{eq:hilbert-mumford-filtration} together generalise~\cite[Proposition 3.4]{BehrendNoohi} to arbitrary hyperquiver representations.

\begin{proposition}
Let $\boldsymbol{\theta} \in \mathbb{Z}^{|\mathcal{V}|}$ be a stability paramater such that
$$\sum_{v \in \mathcal{V}} \theta_v d_v \gamma_v = 0, \text{ for all }\boldsymbol{\gamma} \text{ s.t. } M\gamma \geq 0.$$
Then $\boldsymbol{\phi} \in R(\mathcal{H},\mathbf{d})$ is $\boldsymbol{\theta}$-semistable if and only if for all $\mathbb{Q}$-filtrations $(F^\bullet W,\boldsymbol{\ell})$ admissible for $\boldsymbol\phi$, we have
$$\sum_{v \in \mathcal{V}} \theta_v w_v \geq 0$$
where $w_v := \sum_{k=1}^m \ell_k \dim(F^kW_v / F^{k-1}W_v)$. Equivalently, if $\boldsymbol{\eta} \in \mathbb{Q}_{\geq 0}^{|\mathcal{E}|}$ is a non-negative stability parameter on edges, then $\boldsymbol\phi$ is $\boldsymbol\eta$-semistable if and only if for all $\mathbb{Q}$-filtrations $(F^\bullet W,\boldsymbol{\ell})$ admissible for $\boldsymbol\phi$, we have
\begin{equation*}
\sum_{e \in \mathcal{E}} \eta_e \left(\sum_{t \in t(e)} \frac{w_t}{d_t} - \sum_{s \in s(e)} \frac{w_s}{d_s}\right) \geq 0. \eqdiamond  
\end{equation*}
\end{proposition}

\begin{example}
Let $\mathcal{H} = (\mathcal{V},\mathcal{E})$ be the hyperquiver with hyperedge $(1,2) \rightarrow 3$ and consider a non-zero representation $\phi_e: U_1 \otimes V_1 \rightarrow V_2$. Here, $M = (-1,-1 ,1)$ and its kernel is spanned by $(1,1,2)^T$ and $(0,1,1)^T$, corresponding to the actions of the central 1-PSs $\Gamma_1$ and $\Gamma_2$. If $\chi_\theta$ is a character for stability parameters $(\theta_v)_{v \in \mathcal{V}}$, then the descent conditions \eqref{eq:character-descent} give us
\begin{align*}
    \langle \chi_\theta, \Gamma_1 \rangle &= \theta_1^Ud_1^U + \theta_1^Vd_1^V + 2\theta_2^Vd_2^V = 0 \\
    \langle \chi_\theta, \Gamma_2 \rangle &= \theta_1^Vd_1^V + \theta_2^Vd_2^V = 0 
\end{align*}
Consider the central 1-PS $\Delta = (-1,-1,-1)^T$. Since $M\Delta = 1 \ge 0$, the condition \eqref{eq:farkas-lemma-stability} yields
\begin{align*}
    \langle \chi_\theta, \Delta \rangle &=- \theta_1^Ud_1^U - \theta_1^Vd_1^V - \theta_2^Vd_2^V = - \theta_1^Ud_1^U \ge 0
\end{align*}
Since we may rescale the stability parameters $(\theta_v)_{v \in \mathcal{V}}$ by a positive rational number without changing the inequality \eqref{eq:hilbert-mumford-filtration}, we may assume $\theta_1^Ud_1^U = -1$. Then solving this system of equations yields
\begin{equation*}
    \theta_1^U = -\frac{1}{d_1^U}, \quad \theta_1^V = -\frac{1}{d_1^V}, \quad \theta_2^V = \frac{1}{d_2^V}
\end{equation*}
and substituting these into \eqref{eq:hilbert-mumford-filtration}, we have
\begin{equation*}
    \frac{w_2^V}{d_2^V} - \frac{w_1^U}{d_1^U} - \frac{w_1^V}{d_1^V} \geq 0
\end{equation*}
which is precisely the slope stability condition \eqref{eq:slope-stability-condition-eta} with $\eta_e = 1 > 0$.
\end{example}

\begin{example}
Let $\mathcal{H} = (\mathcal{V},\mathcal{E})$ be the hyperquiver with hyperedge $(1,1) \rightarrow 1$ and consider a non-zero representation $\phi_e: V \otimes V \rightarrow V$ representing an associative $k$-algebra. Here, $M = -1$ and its kernel is $\{0\}$. Thus, \eqref{eq:character-descent} is trivially satisfied. The central 1-PS $\gamma = -1$ satisfies $M\gamma = 1 \geq 0$, so \eqref{eq:farkas-lemma-stability} implies $-a_1 \ge 0$, and thus $\theta \le 0$. To avoid a trivially empty semistable locus, we must have a stability parameter $\theta < 0$. This forces $\eta > 0$. The condition \eqref{eq:slope-stability-condition-eta} reduces to
\begin{equation*}
    -\frac{w}{d} \geq 0
\end{equation*}
for all admissible filtrations. The representation $(V,\phi_e)$ is semistable if and only if it is a semisimple algebra, i.e., contains no non-zero nilpotent ideals \cite[Theorem 1.1]{jones2024canonicalfiltrationsfinitedimensionalalgebras}.
\end{example}

\begin{example}{\label{ex:hyperquiver-3-step}}
Let $k = \mathbb{C}$ and consider the hyperquiver representation in Figure \ref{fig:no-2-step-filtration}, where $W$ is a 3-dimensional vector space with basis $\{x,y,z\}$, given by

\begin{equation*}
\begin{gathered}
	A: W \longrightarrow W \\
	A(x) = x, \quad A(y) = 2y, \quad A(z) = 3z,
\end{gathered}
\end{equation*}
\medskip
\begin{equation*}
\begin{gathered}
	B: W^{\otimes 3} \longrightarrow W^{\otimes 2} \\
	B(x \otimes z \otimes z) = y \otimes y, \quad 0 \text{ on all other basis vectors},
\end{gathered}
\end{equation*}
\medskip
\begin{equation*}
\begin{gathered}
	C: W^{\otimes 2} \longrightarrow W^{\otimes 3}\\
	C(y \otimes y) = x \otimes z \otimes z, \quad 0 \text{ on all other basis vectors},
\end{gathered}
\end{equation*}
\begin{equation*}
\begin{gathered}
	D: W \longrightarrow W^{\otimes 2} \\
	D(x) = x \otimes x + z \otimes z, \quad D(y) = D(z) = 0.
\end{gathered}
\end{equation*}

\medskip

\begin{figure}
    \centering
\begin{tikzpicture}[
    scale=0.9,
    strand/.style={line width=1.6pt, line cap=butt, line join=round},
    hqbarb/.style={postaction={decorate, decoration={markings,
      mark=at position #1 with {
        \draw[line width=1.6pt, line cap=butt]
          (-9.56pt, 3.12pt) .. controls (-6.48pt, 1.91pt) and (-3.24pt, 0.77pt) .. (0pt,0pt)
                            .. controls (-3.24pt,-0.77pt) and (-6.48pt,-1.91pt) .. (-9.56pt,-3.12pt);
      }}}},
  ]
  \begin{scope}[rotate=0, hqblue, strand]
    \draw (0,0) .. controls (1.567,-1.160) and (3.70,-1.95) .. (3.70,0);
    \draw (0,0) .. controls (1.745,-0.870) and (3.70,-1.40) .. (3.70,0);
    \draw (0,0) .. controls (1.870,-0.554) and (3.70,-0.85) .. (3.70,0);
    \draw[hqbarb=0.85] (3.70,0) .. controls (3.70, 0.50) and (1.944, 0.153) .. (0,0);
    \draw[hqbarb=0.85] (3.70,0) .. controls (3.70, 1.95) and (1.567, 1.160) .. (0,0);
  \end{scope}
  \begin{scope}[rotate=88, hqorange, strand]
    \draw (0,0) .. controls (1.700, 0.848) and (3.50, 1.80) .. (3.50,0);
    \draw[hqbarb=0.85] (3.50,0) .. controls (3.50,-0.55) and (1.891,-0.182) .. (0,0);
    \draw[hqbarb=0.85] (3.50,0) .. controls (3.50,-1.80) and (1.700,-0.848) .. (0,0);
  \end{scope}
  \begin{scope}[rotate=182, hqgreen, strand]
    \draw (0,0) .. controls (1.379,-1.379) and (3.70,-1.95) .. (3.70,0);
    \draw (0,0) .. controls (1.635,-1.062) and (3.70,-1.35) .. (3.70,0);
    \draw[hqbarb=0.85] (3.70,0) .. controls (3.70, 0.45) and (1.950, 0.034) .. (0,0);
    \draw[hqbarb=0.85] (3.70,0) .. controls (3.70, 1.15) and (1.795, 0.762) .. (0,0);
    \draw[hqbarb=0.85] (3.70,0) .. controls (3.70, 1.95) and (1.379, 1.379) .. (0,0);
  \end{scope}
  \begin{scope}[rotate=268, hqred, strand]
    \draw[postaction={decorate, decoration={markings,
      mark=at position 0.999 with {
        \draw[line width=1.6pt, line cap=butt]
          (-13.76pt, 3.12pt) .. controls (-10.68pt, 1.91pt) and (-7.44pt, 0.77pt) .. (-4.20pt,0pt)
                             .. controls (-7.44pt,-0.77pt) and (-10.68pt,-1.91pt) .. (-13.76pt,-3.12pt);
      }}}]
      (0,0) .. controls (1.449,-0.388) and (3.00,-1.10) .. (3.00,0)
            .. controls (3.00, 1.10) and (1.449, 0.388) .. (0,0);
  \end{scope}
  \filldraw[hqvert, draw=black, line width=0.8pt] (0,0) circle (4pt);
  \node at (302:1.15) {\phantom{h}};
\end{tikzpicture}
    \caption{An unstable hyperquiver representation with a destabilising 3-step filtration and no 2-step filtrations detecting instability}
    \label{fig:no-2-step-filtration}
\end{figure}
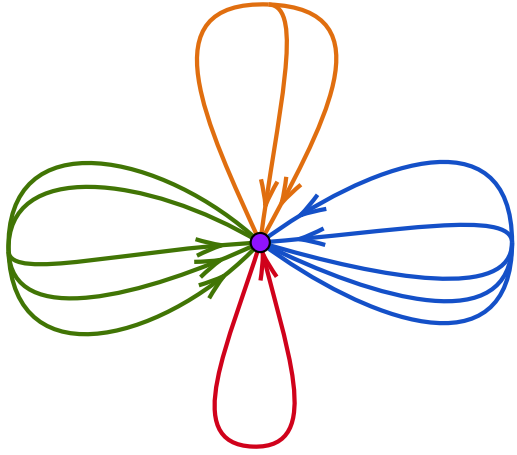

\noindent We show that for $\theta = 1$, i.e., $\chi_\theta(g) = \mathrm{det}(g)$ for $g \in \mathrm{GL}(W)$, this representation has a destabilising admissible 3-step filtration, while all of its admissible 2-step filtrations satisfy the Hilbert-Mumford inequality and therefore cannot detect instability. This demonstrates that instability of hyperquiver representations cannot be tested on 2-step filtrations alone, unlike for quiver representations.

We first characterise every admissible filtration. Any admissible filtration $F^\bullet W = (F^pW)_{p=0}^m$ must satisfy $m \leq 3$, since $\dim(W) = 3$, and $A(F^pW) \subseteq F^pW$, along with similar conditions for $B$, $C$, and $D$. This shows that every $F^pW$ is an $A$-invariant subspace of $W$, and as $A$ has three distinct eigenvalues corresponding to the basis $\{x,y,z\}$, every $F^pW$ must be the span of a subset of this basis. Since a 1-PS $\lambda : \mathbb{C}^* \rightarrow \mathrm{GL}(W)$ corresponding to it acts with some weight on each $F^pW$, we may assume without loss of generality that $\lambda(t) = \mathrm{diag}(t^a,t^b,t^c)$ acts diagonally in this basis. Conjugating the edge maps by $\lambda$ add weights in the targets and subtracts weights in the sources, and admissibility of the filtration is equivalent to non-negativity of the sum of these weights. Thus we obtain the following inequalities: none for $A$ as it is diagonal, and for the remaining maps we have

\[
\llap{$\begin{array}{@{}l@{\;\;} r@{\;}c@{\;}l@{}}
  \strut B: & x \otimes z \otimes z & \mapsto & y \otimes y \\[2pt]
  \strut C: & y \otimes y & \mapsto & x \otimes z \otimes z \\[2pt]
  \multirow{2}{*}{$D:$} & \strut x & \mapsto & x \otimes x \\[2pt]
   & \strut x & \mapsto & z \otimes z
\end{array}$\;}
\begin{array}{@{}c@{}} \strut\iff \\[2pt] \strut\iff \\[2pt] \strut\iff \\[2pt] \strut\iff \end{array}
\rlap{\;\;$\begin{array}{@{}r@{\;}c@{\;}l@{}}
  \strut 2b - (a+2c) & \geq & 0 \\[2pt]
  \strut (a+2c) - 2b & \geq & 0 \\[2pt]
  \strut 2a - a & \geq & 0 \\[2pt]
  \strut 2c - a & \geq & 0
\end{array}$}
\]
Combining these inequalities we have $a - 2b + 2c = 0$, $a \geq 0$, and $2c - a \geq 0$. Rewriting $\alpha := \frac{2c - a}{2} \geq 0$ and $\beta := c \geq 0$, we have $a = -2\alpha + 2\beta$ and $b = -\alpha + 2\beta$, so that $(a,b,c) = \alpha(-2,-1,0) + \beta(2,2,1)$. Thus, the cone of admissible rational 1-PSs is $\mathbb{Q}_{\geq 0}\cdot(-2,-1,0) + \mathbb{Q}_{\geq 0}\cdot(2,2,1)$

Since $m \leq 3$, an admissible 2-step filtration corresponds to a 1-PS $\lambda$ where exactly two of $a$, $b$, and $c$ coincide. If $a = b$, then $\alpha = 0$, giving the ray spanned by $(2,2,1)$. If $b = c$, then $\alpha = \beta$, giving us the ray spanned by $(0,1,1)$. If $a = c$, we have $\beta = 2\alpha$, giving us the ray spanned by $(2,3,2)$. On all three points, we have $(\chi_\theta, \lambda) = a + b + c > 0$, which does not violate the Hilbert-Mumford inequality and thus does not detect instability. However, with $\beta = 0$ on the ray spanned by $(-2,-1,0)$, we have filtrations with $m = 3$ steps as there are three distinct weights $-2$, $-1$, and 0, and the filtration corresponding to this point is destabilising since $(\chi_\theta, \lambda) = (-2) + (-1) + 0 = -3 < 0$, violating the Hilbert-Mumford Criterion.
\end{example}

\section{The Harder-Narasimhan filtration}\label{sec:hn-filtration}

\subsection{Instability}{\label{sec:instability}}

We return to reviewing background from affine GIT. In Section~\ref{sec:affine-GIT}, we studied semistable points. In this section, we study unstable points. As before, our setting is a reductive group $G$ acting on a vector space $V$ together with a $G$-linearisation of the trivial line bundle $L$ over $V$ determined by a character $\chi : G \rightarrow k^*$.

\begin{definition}
Let $\lambda \in X_*(G)$ be a 1-PS of $G$. We define the \textbf{parabolic subgroup} $P(\lambda)$ associated to $\lambda$ to be
\begin{equation}{\label{eq:parabolic-subgroup-for-1-PS}}
P(\lambda) := \left\{g \in G \;\middle|\; \lim_{t \rightarrow 0} \lambda(t) g  \lambda(t)^{-1} \text{ exists in }G\right\},    
\end{equation}
and the \textbf{Levi subgroup} $L(\lambda)$ associated to $\lambda$ to be
\begin{equation}{\label{eq:levi-subgroup}}
L(\lambda) := Z_G(\lambda) = \left\{g \in G \;\middle|\; \lambda(t)g\lambda(t)^{-1} = g, \,\forall t \in k^*\right\}.
\end{equation}
We define $\pi_\lambda : P(\lambda) \rightarrow L(\lambda)$ to be the homomorphism $g \mapsto \lim_{t \rightarrow 0} \lambda(t) g  \lambda(t)^{-1}$, and the \textbf{unipotent subgroup} $U(\lambda) := \mathrm{ker}(\pi_{\lambda})$.
\end{definition}

\begin{remark}{\label{rem:levi-torus-1-PS}}
The unipotent subgroup $U(\lambda)$ is equal to the unipotent radical of $P(\lambda)$, and we have a Levi decomposition $P(\lambda) = U(\lambda) \rtimes L(\lambda)$ where $L(\lambda)$ acts by conjugation on $U(\lambda)$. If $\lambda \in X_*(G)$ is a 1-PS of $G$, then since $\lambda$ commutes with itself, we also have $\lambda \in X_*(L(\lambda))$ a 1-PS of $L(\lambda)$. Additionally, if $T$ is a maximal torus of $G$ containing $\lambda$, then since $T$ is abelian, $\lambda \in X_*(T)$ is also a 1-PS of $T \subseteq L(\lambda)$.
\end{remark}

As we saw in Section~\ref{sec:semistability-of-hyperquivers}, if $G = \prod_{v \in \mathcal{V}} \mathrm{GL}(W_v)$, then a 1-PS $\lambda \in X_*(G)$ determines a weight space decomposition described by a filtration $F^\bullet W = (F^\bullet W_v)_{v \in \mathcal{V}}$, and the subgroup preserving the filtration $F^\bullet W$ is $P(F^\bullet W)$ in \eqref{eq:parabolic-for-1-PS-intro}. In this setting, $P(F^\bullet W)$ is the same as the parabolic subgroup $P(\lambda)$ in \eqref{eq:parabolic-subgroup-for-1-PS}. When written in coordinates by choosing bases for each $W_v$, the components $g_v$ of $\boldsymbol{g} \in G$ become block upper triangular matrices whose diagonal blocks have sizes $\dim W_v^{(i)}$. The unipotent subgroup $U(\lambda)$ consists of block upper unitriangular matrices, i.e., block upper triangular matrices whose diagonal blocks are the identity. Furthermore, the map $\pi_\lambda$ sends every entry strictly above the block diagonal of an element of $P(\lambda)$ to 0. Thus, the Levi subgroup $L(\lambda)$ consists of block diagonal matrices.

We would like to measure the degree to which a point is unstable, and can do that as follows. Let $T$ be a maximal torus in $G$. The Weyl group of $G$ with respect to $T$ is $W_G(T) := N_G(T)/T$, and $w \in W_G(T)$ acts on $T$ by conjugation, hence on 1-PSs $\lambda \in X_*(T)$ by conjugation $\lambda(t) \mapsto w \cdot \lambda(t)$, and on characters $\chi \in X^*(T)$ by pre-conjugation $\chi(\tau) \mapsto \chi(w^{-1} \cdot \tau)$. One can equip the 1-PSs $X_*(T)$ of $T$ with a $W_G(T)$-invariant symmetric positive-definite bilinear form $\langle \cdot, \cdot \rangle : X_*(T) \times X_*(T) \rightarrow \mathbb{Z}$, meaning for any $w \in W_G(T)$ and $\lambda,\lambda' \in X_*(T)$, we have $\langle w\cdot\lambda,w\cdot\lambda'\rangle = \langle \lambda, \lambda'\rangle$. It follows that the Weyl-invariant function $\|\cdot\| : X_*(T) \rightarrow \mathbb{R}_{\geq 0}$ defined by $\|\lambda\| := \langle \lambda, \lambda \rangle^{\frac{1}{2}}$ extends uniquely to a $G$-conjugation invariant function $\|\cdot\| : X_*(G) \rightarrow \mathbb{R}_{\geq 0}$ called a \textit{length function} for 1-PSs of $G$. Similar to the length function, the canonical map $(\cdot,\cdot) : X^*(G) \times X_*(G) \rightarrow \mathbb{Z}$ is $G$-conjugation invariant in its second argument, i.e., $(\chi,g\lambda g^{-1}) = (\chi,\lambda)$ for all $g \in G$. Both also satisfy  $(\chi,\lambda^n) = n(\chi,\lambda)$ and $\|\lambda^n\| = n\|\lambda\|$ for all $n \geq 1$, so their ratio is unchanged when scaling $\lambda$ by a positive integer.

Having such a form and length function allows us to assign canonical 1-PSs to unstable points and, when the action of $G$ on $V$ is linear, decompose the set of unstable points in a canonical manner.

\begin{theorem}[Affine Kempf Instability Theorem {\cite[Theorem 6]{kempfInstability}}]
Let $G$ be a reductive group acting on a vector space $V$, and let $\mathscr{L}$ be the $G$-linearisation corresponding to a character $\chi \in X^*(G)$. Let $\|\cdot\| : X_*(G) \rightarrow \mathbb{R}_{\geq 0}$ be a length function induced by a $W_G(T)$-invariant symmetric positive-definite bilinear form $\langle \cdot, \cdot \rangle : X_*(T) \times X_*(T) \rightarrow \mathbb{Z}$, for some maximal torus $T$ of $G$. If $v \in V$ is an unstable point, then there exists a $\lambda \in X_*(G)$, unique up to scaling by a positive integer and conjugation by the parabolic subgroup $P(\lambda)$ of $G$ associated to $\lambda$, minimising
\begin{equation}{\label{eq:kempf-functional}}
    K_\chi(\lambda) := \frac{(\chi, \lambda)}{\|\lambda\|}
\end{equation}
over $\lambda \in C_v$. In particular, there exists an indivisible such $\lambda$ unique up to conjugation by $P(\lambda)$.
\end{theorem}

Recall that a 1-PS $\lambda \in X_*(G)$ is called \textbf{indivisible} if there does not exist a 1-PS $\mu \in X_*(G)$ and an $n > 1$ such that $\lambda(t) = \mu(t^n)$ for all $t \in k^*$. The function \eqref{eq:kempf-functional} is called the \textbf{Kempf functional}.

\begin{definition}
For an unstable point $v \in V$, let $M_\chi(v) := \inf_{\substack{\lambda \in C_v}} K_\chi(\lambda)$ be the minimum. We call a 1-PS $\lambda \in C_v$ minimising the Kempf functional $K_\chi(\lambda) = M_\chi(v)$ a \textbf{maximally destabilising 1-PS} for $v$. The set of indivisible maximally destabilizing 1-PSs for $v$ is denoted by $\Lambda_\chi(v) \subseteq C_v$.    
\end{definition}

\begin{definition}{\label{def:associated-graded-affine-GIT}}
Let $v \in V$ and $\lambda \in C_v$. We call the point $\mathrm{gr}_\lambda(v) := \lim_{t \rightarrow 0} \lambda(t) \cdot v$ the \textbf{associated graded} point of $v$ with respect to $\lambda$.
\end{definition}

\begin{definition}{\label{def:blade-limit-set-sweep}}
Let $\lambda \in X_*(G)$ be an indivisible 1-PS. We define the \textbf{admissible locus} of $\lambda$ to be
$$V_+^\lambda := \left\{v \in V \;\middle|\; \lim_{t \rightarrow 0} \lambda(t)\cdot v \text{ exists}\right\}$$
and the \textbf{fixed point locus} of $\lambda$ to be
$$V^\lambda := \left\{v \in V \;\middle|\; \lambda(t)\cdot v = v, \, \forall t \in k^*\right\}$$
Let $p_\lambda : V_+^\lambda \rightarrow V^\lambda$ be the \textbf{retraction map} sending $v \mapsto \lim_{t \rightarrow 0} \lambda(t) \cdot v$. We define the \textbf{blade} of $\lambda$ as
\begin{equation}{\label{eq:S-lambda}}
    S_\lambda := \left\{v \in V^{\mathscr{L},\mathrm{u}} \;\middle|\;\lambda \in \Lambda_\chi(v)\right\},
\end{equation}
and the \textbf{limit set} of $\lambda$ to be the image
\begin{equation}{\label{eq:limit-set}}
    Z_\lambda := p_\lambda(S_\lambda).
\end{equation}
The \textbf{Hesselink stratum} associated to $\lambda$ is defined to be
\begin{equation}{\label{eq:hesselink-stata}}
    S_{[\lambda]} := G \cdot S_\lambda
\end{equation}
and the \textbf{Hesselink index set} is
\begin{equation}{\label{eq:hesselink-index-set}}
    \mathcal{B} := \left\{[\lambda] \,:\, \exists\, v \in V^{\mathscr{L},\mathrm{u}} \text{ s.t. }\lambda \in \Lambda_\chi(v)\right\}
\end{equation}
where $[\lambda]$ denotes the $G$-conjugacy class of $\lambda$.
\end{definition}

The loci $V_+^\lambda$ and $V^\lambda$ are closed subsets of $V$ and the blades $S_\lambda$ are locally closed subvarieties of $V$ and equal $S_\lambda = p_\lambda^{-1}(Z_\lambda)$ \cite[Lemma 2.19]{Hoskins2012Stratifications}. There is also a strict partial ordering $<$ on $\mathcal{B}$ where $[\lambda] < [\lambda']$ if $K_\chi(\lambda') < K_\chi(\lambda)$. When the action of $G$ on $V$ is linear, a 1-PS $\lambda \in X_*(G)$ splits $V$ into weight spaces $V = \bigoplus_{i \in \mathbb{Z}} V^i$ where $V^i = \{v \in V \;|\; \lambda(t) v = t^i v, \, \forall t \in k^*\}$. Then the admissible locus is $V_+^\lambda = \bigoplus_{i \geq 0} V^i$ and the fixed point locus is $V^\lambda = V^0$.

\begin{theorem}[Hesselink-Kirwan-Kempf-Ness Stratification {\cite[Theorem 2.16]{Hoskins2012Stratifications}}]
Let $G$ be a reductive group acting linearly on a vector space $V$, and let $\mathscr{L}$ be the $G$-linearisation corresponding to a character $\chi \in X^*(G)$. Fix maximal torus $T$ of $G$ and a $W_G(T)$-invariant symmetric positive-definite bilinear form $\langle \cdot, \cdot \rangle : X_*(T) \times X_*(T) \rightarrow \mathbb{Z}$. Then the Hesselink index set $\mathcal{B}$ is finite, the Hesselink strata $S_{[\lambda]}$ are all disjoint, and there exists a stratification
$$V = V^{\mathscr{L},\mathrm{ss}} \; \sqcup \; \bigsqcup_{[\lambda] \in \mathcal{B}} S_{[\lambda]}.$$
with $\partial S_{[\lambda]} \cap S_{[\lambda']} \neq \varnothing$ only if $[\lambda] < [\lambda']$. Each $S_{[\lambda]}$ is a smooth locally closed $G$-invariant subvariety of $V$.
\end{theorem}

As before in Section~\ref{sec:semistability-of-hyperquivers}, we want to consider, more generally, rational (co)characters of $G$. We can extend the form $\langle \cdot, \cdot \rangle : X_*(T) \times X_*(T) \rightarrow \mathbb{Z}$ to a Weyl-invariant inner product $\langle \cdot, \cdot \rangle : X_*(T)_{\mathbb{Q}} \times X_*(T)_{\mathbb{Q}} \rightarrow \mathbb{Q}$ where $X_*(T)_{\mathbb{Q}} = X_*(T) \otimes \mathbb{Q}$, giving us a length function $\|\cdot\| : X_*(G)_{\mathbb{Q}} \rightarrow \mathbb{R}_{\geq 0}$. Together with the extended canonical map $(\cdot,\cdot) : X^*(G)_{\mathbb{Q}} \times X_*(G)_{\mathbb{Q}} \rightarrow \mathbb{Q}$, we can therefore minimise the Kempf functional \eqref{eq:kempf-functional} over rational 1-PSs.

Let $[\lambda] \in \mathcal{B}$. We can choose a representative $\lambda \in X_*(T)_{\mathbb{Q}}$ in $T$, and by Remark~\ref{rem:levi-torus-1-PS}, $\lambda \in X_*(L(\lambda))_{\mathbb{Q}}$ and it is also a 1-PS of $L(\lambda)$. Let $\lambda^* \in X^*(T)_{\mathbb{Q}}$ be the character dual to $\lambda$ under the inner product $\langle \cdot, \cdot \rangle$. Since $\lambda$ is central in $L(\lambda)$, it is invariant under the action of $W_{L(\lambda)}(T)$, and from this fact, it follows that $\lambda^*$ is $W_{L(\lambda)}(T)$-invariant. By \cite[Proposition 2.3.2]{herrero2026stabilitydisconnectedgroups}, there is a unique extension to a character $\lambda^* \in X^*(L(\lambda))_\mathbb{Q}$. We can therefore consider a new affine GIT problem where $L(\lambda)$ acts on $V^\lambda$ with a linearisation given by a new character. This gives a recursive refinement on the Hesselink strata, which has been studied in \cite{Kirwan2005} and \cite{núñez2023refinedhardernarasimhanfiltrationsmoduli}.

\begin{proposition}[{\cite[Proposition 2.18]{Hoskins2012Stratifications}}]{\label{prop:shifted-linearisation}}
Let $[\lambda] \in \mathcal{B}$ and choose a represenative $\lambda \in X_*(T)$. Let $L(\lambda)$ act linearly on the vector space $V^\lambda$ with a linearisation $\mathscr{L}_\lambda$ given by the rational character $\chi_\lambda := \|\lambda\|^2\chi - (\chi,\lambda)\lambda^*$ of $L(\lambda)$. Then $(V^\lambda)^{\mathscr{L}_\lambda,\mathrm{ss}} = Z_\lambda$. In particular, if $v \in V^{\mathscr{L},\mathrm{u}}$ and $\lambda \in \Lambda_\chi(v)$, then the associated graded point $\mathrm{gr}_\lambda(v) \in V^\lambda$ is semistable with respect to linearisation given by $\chi_\lambda$.
\end{proposition}

\subsection{The Harder-Narasimhan filtration}{\label{sec:harder-narasimhan}}

We translate instability in affine GIT for our hyperquiver setting of $G = \prod_{v \in \mathcal{V}} \mathrm{GL}(W_v)$ acting on $R(\mathcal{H},\mathbf{d})$ with respect to a stability parameter $\boldsymbol{\theta} \in \mathbb{Q}^{|\mathcal{V}|}$.

\begin{definition}{\label{def:induced-lambda}}
A $\mathbb{Q}$-filtration $(F^\bullet W,\boldsymbol{\ell})$ is said to be \textit{induced} by a rational 1-PS $\lambda \in X_*(G)_{\mathbb{Q}}$ if $\lambda$ induces a weight space decomposition $W_v = \bigoplus_{k=1}^m W_v^{(\ell_k)}$ with weights $\ell_1 > \cdots > \ell_m$ and $F^kW_v = W_v^{(\ell_k)} \oplus F^{k-1}W_v$.
\end{definition}

\begin{definition}
\label{def:refined-hyperquiver}
Let $\boldsymbol{\phi} \in R(\mathcal{H},\mathbf{d})$ be a representation and let $(F^\bullet W,\boldsymbol{\ell})$ be a $\mathbb{Q}$-filtration admissible for $\boldsymbol{\phi}$, induced by a rational 1-PS $\lambda \in X_*(G)_{\mathbb{Q}}$. The \textbf{refined hyperquiver} $\mathcal{H}_{F}$ is the hyperquiver with
\begin{enumerate}
    \item vertex set $\mathcal{V} \times [m]$
    \item an edge $c = (e,I,J)$ for every $e \in \mathcal{E}$, $I : s(e) \rightarrow [m]$, and $J : t(e) \rightarrow [m]$, having source vertices $((s,I(s)))_{s \in s(e)}$ and target vertices $((t,J(t)))_{t \in t(e)}$
\end{enumerate}
The \textbf{associated graded representation} $\mathrm{gr}_{F}(\boldsymbol{\phi})$ is the representation of the hyperquiver $\mathcal{H}_{F}$ defined as follows.
\begin{enumerate}
    \item it has vector spaces $\mathrm{gr}_kW_v := F^kW_v/F^{k-1}W_v$ on each vertex $(v,k) \in \mathcal{V} \times [m]$
    \item it has component maps
    $$\mathrm{gr}_F(\boldsymbol{\phi})_{(e,I,J)} : \bigotimes_{s \in s(e)} \mathrm{gr}_{I(s)}W_s \rightarrow \bigotimes_{t \in t(e)} \mathrm{gr}_{J(t)}W_t$$
    on each edge $c = (e,I,J)$ for all $e \in \mathcal{E}$, $I : s(e) \rightarrow \{1,\ldots,m\}$, $J : t(e) \rightarrow \{1,\ldots,m\}$, which are identical to those of $\boldsymbol{\phi}$, except when
    $$\sum_{t \in t(e)} \ell_{J(t)} > \sum_{s \in s(e)} \ell_{I(s)}$$
    where the component map is the zero map.
\end{enumerate}
We denote the set of non-zero components $c = (e,I,J)$ of $\mathrm{gr}_F(\boldsymbol{\phi})$ by $\mathcal{C}(\mathrm{gr}_F(\boldsymbol{\phi}))$.
\end{definition}

The associated graded representation $\mathrm{gr}_F(\boldsymbol\phi)$ corresponds to the associated graded point $\mathrm{gr}_\lambda(\boldsymbol{\phi})$. Its component maps $\mathrm{gr}_F(\boldsymbol\phi)_{(e,I,J)}$ come from the weight space decomposition induced by $\lambda$. The shifted stability parameter $\widetilde{\boldsymbol{\theta}}$ is precisely the stability parameter corresponding to the linearisation $\mathscr{L}^\lambda$ given by the character $\chi_\lambda$ in Proposition~\ref{prop:shifted-linearisation}. We now show that, in fact, the associated graded representation is canonical.

\begin{lemma}{\label{lem:associated-graded-independent}}
Let $(F^\bullet W,\boldsymbol{\ell})$ be a $\mathbb{Q}$-filtration.
\begin{enumerate}
    \item Two rational 1-PSs $\lambda,\lambda' \in X_*(G)_{\mathbb{Q}}$ induce the same $\mathbb{Q}$-filtration $(F^\bullet W,\boldsymbol{\ell})$ if and only if $\lambda' = \boldsymbol{p}\lambda \boldsymbol{p}^{-1}$ for some $\boldsymbol{p} \in P(\lambda)$.
    \item The associated graded representation $\mathrm{gr}_F(\boldsymbol\phi)$ is independent of the choice of $\lambda \in X_*(G)_{\mathbb{Q}}$ inducing $(F^\bullet W,\boldsymbol{\ell})$.
\end{enumerate}
\end{lemma}

\begin{proof}
First we prove (1). If $\lambda' = \boldsymbol{p}\lambda \boldsymbol{p}^{-1}$ and $\lambda$ induces $(F^\bullet W,\boldsymbol{\ell})$, then $\lambda'$ induces $(\boldsymbol{p}F^\bullet W,\boldsymbol{\ell})$. But $\boldsymbol{p} \in P(\lambda)$ preserves $F^\bullet W$, so $\lambda'$ induces the same filtration. Conversely, if $\lambda$ and $\lambda'$ induce $(F^\bullet W,\boldsymbol{\ell})$, then we have two complements $F^kW_v = W_v^{(\ell_k)} \oplus F^{k-1}W_v = W_v'^{(\ell_k)} \oplus F^{k-1}W_v$ for all $v \in V$ and $k \in [m]$. Let $u_{v,k} : F^kW_v \rightarrow W_v'^{(\ell_k)}$ be the projection maps and extend them linearly to a linear map $\boldsymbol{u} = (u_v)_{v \in V}$. The map $\boldsymbol{u}$ is invertible as it is block upper unitriangular with blocks $(u_v)_{ij} : W_v^{(\ell_j)} \rightarrow W_v^{(\ell_i)}$, for $1 \leq i \leq j \leq m$. Conjugating $\boldsymbol{u}$ by $\lambda(t)$ scales every block by $t^{\ell_i - \ell_j}$, so if $i < j$, then $\ell_j < \ell_i$ so this block is sent to 0 as $t \rightarrow 0$. Thus, $\lambda(t) \boldsymbol{u} \lambda(t)^{-1} \xrightarrow[]{t \rightarrow 0} \mathrm{id}$ and hence $\boldsymbol{u} \in \mathrm{ker}(\pi_{\lambda}) = U(\lambda)$. Furthermore, for any $w' \in W_v'^{(\ell_k)}$ and $w' = \boldsymbol{u}\cdot w$,
$$\boldsymbol{u}\lambda(t)\boldsymbol{u}^{-1}\cdot w' = \boldsymbol{u}\lambda(t) \cdot w = \boldsymbol{u} \cdot t^{\ell_k}w = t^{\ell_k} \boldsymbol{u} \cdot w = t^{\ell_k} w' = \lambda'(t)\cdot w'.$$
Thus, we can take $\boldsymbol{p} = \boldsymbol{u}$.

Now we prove (2), which has the same hypotheses as (1). Suppose $\lambda$ and $\lambda'$ induce $(F^\bullet W,\boldsymbol{\ell})$. Using the same $\boldsymbol{u}$ as before, we have 
$$\lim_{t \rightarrow 0} \lambda'(t)\cdot \boldsymbol{\phi} = \lim_{t \rightarrow 0} (\boldsymbol{u}\lambda(t)\boldsymbol{u}^{-1})\cdot \boldsymbol{\phi} = \boldsymbol{u}\cdot\lim_{t \rightarrow 0} (\lambda(t)\boldsymbol{u}^{-1}\lambda(t)^{-1})\lambda(t)\cdot \boldsymbol{\phi} = \boldsymbol{u}\cdot\lim_{t \rightarrow 0} \lambda(t)\cdot \boldsymbol{\phi}.$$
Since $\boldsymbol{u}$ induces an identity on $F^{k}W_v/F^{k-1}W_v$, both $\lambda$ and $\lambda'$ give the same map between these quotients.
\end{proof}

\begin{definition}{\label{def:Q-filtration-hn-types}}
Let $(F^\bullet W,\boldsymbol{\ell})$ be a $\mathbb{Q}$-filtration admissible for a representation $\boldsymbol{\phi} \in R(\mathcal{H},\mathbf{d})$.
\begin{enumerate}
	\item For all $v \in \mathcal{V}$ and $k \in [m]$, we define $d_{v,k} = \dim(\mathrm{gr}_kW_v)$
    \item $\boldsymbol{\Theta}$ is the vector $\boldsymbol{\Theta} \in \mathbb{Q}^m$ with entries $\Theta_k  = \sum_{v \in V}\theta_v d_{v,k}$
    \item $D$ is the diagonal matrix $\mathrm{diag}(D_k)_{k=1}^m$, where $D_k = \sum_{v \in V} d_{v,k}$
    \item $A$ is the $|\mathcal{C}(\mathrm{gr}_F(\boldsymbol{\phi}))|\times m$ integer matrix whose entries are 
$$A_{c,k} := |\{ t \in t(e) : J(t) = k \} | - | \{ s \in s(e) : I(s) = k \}|$$
    for $c = (e,I,J) \in \mathcal{C}(\mathrm{gr}(\boldsymbol{\phi}))$ and $k \in [m]$
    \item We define $\|\boldsymbol{\cdot}\|^2_D : \mathbb{R}^m \rightarrow \mathbb{R}_{\geq 0}$ by $\|\boldsymbol{p}\|_D^2 = \sum_{v \in \mathcal{V}} \sum_{k=1}^m p_k^2 d_{v,k}$
    \item We define the \textbf{shifted stability parameter} $\widetilde{\boldsymbol{\theta}} \in \mathbb{Q}^{|\mathcal{V}| \times m}$ for $\mathrm{gr}_F(\boldsymbol\phi)$ to be
\begin{equation}{\label{eq:shifted-stability-parameters}}
    \widetilde{\theta}_{v,k} := \theta_v - \frac{\langle \boldsymbol{\Theta},\boldsymbol{\ell}\rangle}{\|\boldsymbol{\ell}\|_D^2}\ell_k, \quad v \in \mathcal{V}, \; k \in [m]
\end{equation}
\end{enumerate}    
\end{definition}

The shifted stability parameter $\widetilde{\boldsymbol{\theta}}$ corresponds to the rational character $\chi_\lambda$ mentioned in Proposition~\ref{prop:shifted-linearisation}, but scaled by $1/\|\lambda\|^2$.

In order to explicitly find a minimiser of the Kempf functional \eqref{eq:kempf-functional}, we will need to choose a Weyl-invariant inner product and find necessary and sufficient conditions to characterise such a minimizer. We fix a maximal torus $T$ of $G = \prod_{v \in \mathcal{V}} \mathrm{GL}(W_v)$ and choose the standard (Euclidean) inner product $\langle \cdot, \cdot \rangle$ on $X_*(T)_{\mathbb{Q}}$ with respect to a diagonalising basis for $T$. In this case, the Weyl group is the symmetric group $W_G(T) = \prod_{v \in \mathcal{V}} S_{d_{v}}$ permuting the coordinates of $\lambda_v \in X_*(T_v)$ within each $v \in \mathcal{V}$, and the inner product $\langle \cdot, \cdot \rangle$ is clearly invariant under this action. We will use the following characteristion of minimisers from optimisation theory.

\begin{theorem}[KKT Conditions for a Quadratic Program, \label{thm:kkt}{\cite[Section 16.4]{luenberger2008linear}}]
Consider an optimization problem of the form
\begin{equation}
    \begin{split}
        \mathrm{minimize} \quad & \frac{1}{2}\mathbf{x}^\top G \mathbf{x} + \mathbf{c}^\top \mathbf{x} \\
        \mathrm{subject}\;\mathrm{to} \quad & \mathbf{a}_i^\top\mathbf{x} = \alpha_i, \quad \forall i \in [m] \\
        & \mathbf{b}_j^\top\mathbf{x} \ge \beta_j, \quad \forall j \in [n]
    \end{split}
    \label{eq:kkt-quadratic-program}
\end{equation}
where $\mathbf{x},\mathbf{c},\mathbf{a}_i, \mathbf{b}_j \in \mathbb{R}^N$, $\alpha_i, \beta_j \in \mathbb{R}$, and $G \in \mathbb{R}^{N \times N}$ is positive definite. A point $\mathbf{x}^* \in \mathbb{R}^N$ is the unique global minimiser for \eqref{eq:kkt-quadratic-program} if and only if there exist $\mu_i, \nu_j \in \mathbb{R}$ such that:
\begin{enumerate}
    \item $G\mathbf{x}^* + \mathbf{c} + \sum_{i=1}^m \mu_i \mathbf{a}_i - \sum_{j=1}^n \nu_j \mathbf{b}_j = \mathbf{0}$
    \item $\mathbf{a}_i^\top\mathbf{x}^* = \alpha_i$ for $i \in [m]$ and $\mathbf{b}_j^\top\mathbf{x}^* \ge \beta_j$ for $j \in [n]$
    \item $\nu_j \ge 0$ for $j \in [n]$
    \item $\nu_j \cdot (\mathbf{b}_j^\top \mathbf{x}^* - \beta_j) = 0$ for $j \in [n]$
\end{enumerate}
\end{theorem}

Thus, in our setting, $G = I$ and $\mathbf{c} = 0$ corresponds to the Euclidean inner product. We are now ready to state and prove our main theorem.

\begin{theorem}\label{thm:hn-filtration}
Let $\mathcal{H} = (\mathcal{V},\mathcal{E})$ be a hyperquiver, $\boldsymbol{\theta} \in \mathbb{Q}^{|\mathcal{V}|}$ be a stability parameter, and $\boldsymbol{\phi} \in R(\mathcal{H},\mathbf{d})$ be a $\boldsymbol{\theta}$-unstable representation. Fix a maximal torus $T$ of $G$ and equip $X_*(T)_{\mathbb{Q}}$ with the standard inner product. Then there
exists a unique $\mathbb{Q}$-filtration
$\mathrm{HN}(\boldsymbol{\phi}) = (F^\bullet W,\boldsymbol{\ell})$ admissible for $\boldsymbol{\phi}$, the \emph{Harder-Narasimhan filtration}, such that
\begin{enumerate}
  \item[(a)] $\langle \boldsymbol\Theta,\boldsymbol{\ell}\rangle = -1$
  \item[(b)] the associated graded representation $\mathrm{gr}_{\mathrm{HN}}(\boldsymbol{\phi})$ of the refined hyperquiver $\mathcal{H}_{F}$ is $\widetilde{\boldsymbol{\theta}}$-semistable for the
  shifted stability parameters $\tilde{\boldsymbol{\theta}}$ given in Definition~\ref{def:Q-filtration-hn-types}
\end{enumerate}
Furthermore, the weights $\ell_k$ are determined by the formula 
\begin{equation}{\label{eq:HN-weights-formula-thm}}
\boldsymbol{\ell} = \boldsymbol{p}/\|\boldsymbol{p}\|_D^2 \quad \text{ for } \quad \boldsymbol{p} = -D^{-1} \left( I - A^\top L^\dagger A D^{-1}  \right)\boldsymbol{\Theta}
\end{equation}
where $\boldsymbol\Theta$, $D$, $A$, $\|\cdot\|^2_D$ are given in Definition~\ref{def:Q-filtration-hn-types},  $L := AD^{-1}A^\top$, and $L^\dagger$ is the Moore-Penrose inverse of $L$. They all only depend on the hyperquiver representation $\mathrm{gr}_\mathrm{HN}(\boldsymbol\phi)$.
\end{theorem}

\begin{proof}

First we prove the existence of $\mathrm{HN}(\boldsymbol{\phi})$. By the Affine Kempf Instability Theorem, there exists a rational 1-PS $\lambda \in X_*(G)_{\mathbb{Q}}$ minimising the Kempf functional \eqref{eq:kempf-functional}, unique up to positive rational scaling and conjugation by $P(\lambda)$. We rescale it to satisfy
$$(\chi_\theta, \lambda) = \sum_{v \in \mathcal{V}} \theta_v\sum_{k=1}^m \ell_k d_{v,k} = \langle \boldsymbol{\Theta}, \boldsymbol{\ell} \rangle =  -1$$
We define $\mathrm{HN}(\boldsymbol{\phi})$ to be the admissible $\mathbb{Q}$-filtration induced by the rational 1-PS $\lambda$. By Lemma~\ref{lem:associated-graded-independent}(2), $\mathrm{HN}(\boldsymbol{\phi})$ is well-defined and independent of the choice of $\lambda$.

First we prove the formula \eqref{eq:HN-weights-formula-thm}. As all maximal tori are conjugate in $G$, we may assume without loss generality that $T$ constains $\lambda$. As a rational 1-PS of $T$, we may express $\lambda$ to be diagonal with respect to an eigenbasis $\{\boldsymbol{\varepsilon}_{v,i}\}_{i = 1}^{d_v}$ for each $W_v$ for $T$. This identifies $X_*(T)_\mathbb{Q}$ with $\mathbb{Q}^N$ where $N = \sum_{v \in \mathcal{V}} d_v$. Let $\langle \cdot, \cdot \rangle$ be the Euclidean inner product on $\mathbb{Q}^N$. Let $\mathbf{x} = (x_{v,i})$ denote the weights of a rational 1-PS of $T$, $\widehat{\boldsymbol{\theta}} = (\widehat{\theta}_{v,i})$ denote the stability parameter $\widehat{\theta}_{v,i} := \theta_v$, and $\widehat{\boldsymbol{\ell}} = (\widehat{\ell}_{v,i})$ denote the weights of the rational 1-PS $\lambda$. We have $(\chi_{\boldsymbol{\theta}},\lambda) = \langle \widehat{\boldsymbol\theta},\widehat{\boldsymbol{\ell}}\rangle = -1$. Let $b = (e,\widehat{I},\widehat{J})$ where $e \in \mathcal{E}$, $\widehat{I}$ is a function sending $s \in s(e)$ to $\widehat{I}(s) \in [d_s]$, $\widehat{J}$ is a function sending $t \in t(e)$ to $\widehat{J}(t) \in [d_t]$. Let $C_b$ be the coefficient of the component $\boldsymbol{\phi}_b : \bigotimes_{s \in s(e)} \boldsymbol{\varepsilon}_{s,\widehat{I}(s)} \rightarrow \bigotimes_{t \in t(e)} \boldsymbol{\varepsilon}_{t,\widehat{J}(t)}$. Denote by $B_{\boldsymbol{\phi}}$ the set of all $b$ such that $C_b \neq 0$. Let
$$\mathbf{a}_b := \sum_{t \in t(e)} \boldsymbol{e}_{t,\widehat{J}(t)} - \sum_{s \in s(e)} \boldsymbol{e}_{s,\widehat{I}(s)} \in \mathbb{Z}^N$$
where $\boldsymbol{e}_{v,i}$ is a standard basis vector in $\mathbb{Z}^N$. Consider the quadratic program
\begin{equation*}
    \begin{split}
        \mathrm{minimize} \quad & \frac{1}{2}\|\mathbf{x}\|^2 \\
        \mathrm{subject}\;\mathrm{to} \quad & \langle \widehat{\boldsymbol{\theta}},\mathbf{x}\rangle = -1, \\
        & \langle \mathbf{a}_b, \mathbf{x} \rangle \ge 0, \quad \forall b \in B_{\boldsymbol{\phi}}
    \end{split}
\end{equation*}
over $\mathbf{x} \in \mathbb{R}^N$. With $(\chi_{\boldsymbol{\theta}},\lambda) = -1$ fixed, minimising the Kempf functional is equivalent to minimising $\|\lambda\|$, which is equivalent to finding a minimiser $\mathbf{x} \in \mathbb{Q}^N$ for this quadratic program. Consider the rational polyhedron
$$S := \{\mathbf{x} \in \mathbb{R}^N \;|\; 
\langle\widehat{\boldsymbol{\theta}},\mathbf{x}\rangle = -1, \langle \mathbf{a}_b, \mathbf{x} \rangle \ge 0, \, \forall b \in B_{\boldsymbol{\phi}}\}$$
in $\mathbb{R}^N$. By definition, $\widehat{\boldsymbol{\ell}} \in S$ and for every rational point $\mathbf{x} \in S \cap \mathbb{Q}^N$, we have $\frac{1}{2}\|\widehat{\boldsymbol{\ell}}\|^2 \leq \frac{1}{2}\|\mathbf{x}\|^2$. By density of rational points, $\frac{1}{2}\|\widehat{\boldsymbol{\ell}}\|^2 \leq \frac{1}{2}\|\mathbf{x}\|^2$ for all $\mathbf{x} \in S$. As $S$ is convex and $\frac{1}{2}\|\cdot\|^2$ is strictly convex, $\widehat{\boldsymbol{\ell}}$ is the unique minimiser of the quadratic program. Thus, by the Karush-Kuhn-Tucker Conditions, there exists a $\mu \in \mathbb{R}$ and $\nu_b \geq 0$ such that
\begin{equation}{\label{eq:kkt-condition-1}}
    \widehat{\boldsymbol{\ell}} + \mu \widehat{\boldsymbol{\theta}} - \sum_{b \in B_{\boldsymbol{\phi}}} \nu_b \mathbf{a}_b = \boldsymbol0
\end{equation}
and
\begin{equation}{\label{eq:kkt-condition-2}}
    \nu_b \langle \mathbf{a}_b, \widehat{\boldsymbol{\ell}}\rangle = 0, \quad \forall b \in B_{\boldsymbol{\phi}}.
\end{equation}
Taking the inner product of both sides of \eqref{eq:kkt-condition-1} with $\widehat{\boldsymbol{\ell}}$ and using $\langle \widehat{\boldsymbol{\theta}},\widehat{\boldsymbol{\ell}} \rangle = -1$ and \eqref{eq:kkt-condition-2} allows us to determine
$$\mu = \|\widehat{\boldsymbol{\ell}}\|^2 = \|\boldsymbol{\ell}\|_D^2 > 0.$$ Let $c = (e,I,J)$ for $e \in \mathcal{E}$, $I : s(e) \rightarrow [m]$, and $J : t(e) \rightarrow [m]$. Define
$$B_c := \left\{b = (e,\widehat{I},\widehat{J}) \in B_{\boldsymbol{\phi}} \;\middle|\; \widehat{\ell}_{s,\widehat{I}(s)} = \ell_{I(s)}, \,\forall s \in s(e), \; \widehat{\ell}_{t,\widehat{J}(t)} = \ell_{J(t)}, \, \forall t \in t(e) \right\}.$$
Then every $b \in B_{\boldsymbol{\phi}}$ belongs to exactly one $B_c$. Furthermore, if $b \in B_c$ for some $c = (e,I,J)$, then
\begin{equation}{\label{eq:kkt-balanced}}
    \langle \mathbf{a}_b,\widehat{\boldsymbol\ell}\rangle = \sum_{t \in t(e)} \widehat{\ell}_{t,\widehat{J}(t)} - \sum_{s \in s(e)} \widehat{\ell}_{s,\widehat{I}(s)} = \sum_{t \in t(e)} \ell_{{J}(t)} - \sum_{s \in s(e)} \ell_{I(s)}.  
\end{equation}
If $\nu_b > 0$, then $\langle \mathbf{a}_b,\widehat{\boldsymbol\ell}\rangle = 0$ by \eqref{eq:kkt-condition-2}. Hence, $c \in \mathcal{C}(\mathrm{gr}_{\mathrm{HN}}(\boldsymbol\phi))$. Otherwise, $\nu_b = 0$ does not contribute to the sum \eqref{eq:kkt-condition-1}. Thus, \eqref{eq:kkt-condition-1} becomes
\begin{equation}{\label{eq:kkt-condition-3}}
    \widehat{\boldsymbol{\ell}} + \|\boldsymbol{\ell}\|_D^2 \widehat{\boldsymbol{\theta}} - \sum_{c \in \mathcal{C}(\mathrm{gr}_{\mathrm{HN}}(\boldsymbol\phi))} \sum_{b \in B_c} \nu_b \mathbf{a}_b = \boldsymbol0.
\end{equation}
Let $\Pi : \mathbb{R}^N \rightarrow \mathbb{R}^m$ be defined on the basis vectors $\Pi(\boldsymbol{e}_{v,i}) := \boldsymbol{e}_k$ if $\widehat{\ell}_{v,i} = \ell_k$. Then applying $\Pi$ to \eqref{eq:kkt-condition-3}, we have $\Pi(\widehat{\boldsymbol{\ell}}) = D\boldsymbol{\ell}$, $\Pi(\widehat{\boldsymbol{\theta}}) = \boldsymbol{\Theta}$, and $\Pi(\mathbf{a}_b) = \mathbf{a}_c$ where
\begin{equation}{\label{eq:rows-of-A}}
\mathbf{a}_c := \sum_{t \in t(e)} \boldsymbol{e}_{J(t)} - \sum_{s \in s(e)} \boldsymbol{e}_{I(s)} \in \mathbb{Z}^m.    
\end{equation}
The vectors $\mathbf{a}_c$ form the rows of the matrix $A$ given in Definition~\ref{def:Q-filtration-hn-types}. Thus, \eqref{eq:kkt-condition-3} becomes
\begin{equation}{\label{eq:kkt-condition-4}}
    D\boldsymbol{\ell} + \|\boldsymbol{\ell}\|_D^2 \boldsymbol{\Theta} - A^\top\widetilde{\boldsymbol{\nu}} = \boldsymbol0
\end{equation}
where $\widetilde{\nu}_c = \sum_{b \in B_c} \nu_b \geq 0$. Multiply \eqref{eq:kkt-condition-4} by $AD^{-1}$ and we obtain
\begin{equation}{\label{eq:kkt-condition-5}}
     L\widetilde{\boldsymbol{\nu}} = \mathbf{y}
\end{equation}
where the first term has vanished since $A\boldsymbol{\ell} = \boldsymbol{0}$ by \eqref{eq:kkt-balanced}, and we have defined $L := AD^{-1}A^\top$ and $\mathbf{y} := \|\boldsymbol{\ell}\|_D^2 AD^{-1}\boldsymbol{\Theta}$. We see that $\mathrm{ker}(L) = \mathrm{ker}(A^\top)$ since if $A^\top\mathbf{x} = \boldsymbol{0}$, then clearly $L\mathbf{x} = \boldsymbol0$, and if $L\mathbf{x} = \boldsymbol{0}$, then $0 = \mathbf{x}^\top L \mathbf{x} = (A^\top\mathbf{x})^\top D^{-1}(A^\top\mathbf{x})$, and since $D^{-1}$ is positive-definite, $A^\top\mathbf{x} = \boldsymbol{0}$. Let $\widetilde{\boldsymbol{\nu}}' := L^\dagger \mathbf{y}$ where $L^\dagger$ is the Moore-Penrose inverse of $L$ \cite[Section 5.5.2]{GolubVanLoan2013}. Then we see that
$$L\widetilde{\boldsymbol{\nu}}' = LL^\dagger\mathbf{y} = LL^\dagger L\widetilde{\boldsymbol{\nu}} = L\widetilde{\boldsymbol{\nu}} = \mathbf{y}$$
by \eqref{eq:kkt-condition-5} and the property $LL^\dagger L = L$ of the Moore-Penrose inverse. Thus, $\widetilde{\boldsymbol{\nu}}' - \widetilde{\boldsymbol{\nu}} \in \mathrm{ker}(L) = \mathrm{ker}(A^\top)$ so
$$A^\top \widetilde{\boldsymbol{\nu}} = A^\top \widetilde{\boldsymbol{\nu}}'=  \|\boldsymbol{\ell}\|_D^2 A^\top L^\dagger AD^{-1}\boldsymbol{\Theta}.$$
Substituting this $A^\top \widetilde{\boldsymbol{\nu}}$ into \eqref{eq:kkt-condition-4} and rearranging, we obtain the formula
$$\frac{\boldsymbol{\ell}}{\|\boldsymbol{\ell}\|^2_D} = -D^{-1} \left( I - A^\top L^\dagger A D^{-1} \right) \mathbf{\Theta}.$$

Now we prove (b). Note that $\lambda$ is fixed from before. Let $\lambda' \in X_*(L(\lambda))_{\mathbb{Q}}$ be a rational 1-PS of the Levi subgroup $L(\lambda)$, admissible for $\mathrm{gr}_{\mathrm{HN}}(\boldsymbol{\phi})$. Let $T'$ be a maximal torus in $L(\lambda)$ containing both $\lambda$ and $\lambda'$. Then we can simultaneously diagonalise them in $T'$ with respect to a single eigenbasis for each $W_v$. Redoing the same argument as before, by the KKT Conditions, there exist $\mu \in \mathbb{R}$ and $\nu_b \geq 0$ such that
\begin{equation}{\label{eq:kkt-condition-1-gr}}
    \widehat{\boldsymbol{\ell}} + \mu \widehat{\boldsymbol{\theta}} = \sum_{b \in B_{\boldsymbol{\phi}}} \nu_b \mathbf{a}_b
\end{equation}
and
\begin{equation}{\label{eq:kkt-condition-2-gr}}
    \nu_b \langle \mathbf{a}_b, \widehat{\boldsymbol{\ell}}\rangle = 0, \quad  \forall b \in B_{\boldsymbol{\phi}}.
\end{equation}
Recall that we imposed the normalisation $\langle \boldsymbol\Theta, \widehat{\boldsymbol{\ell}}\rangle = -1$. Again we find $\mu = \|\boldsymbol{\ell}\|_D^2 > 0$. Let $\widehat{\boldsymbol{r}} \in \mathbb{Q}^N$ be the weights of the rational 1-PS $\lambda'$. Taking the inner product of \eqref{eq:kkt-condition-1-gr} with $\widehat{\boldsymbol{r}}$ gives
\begin{equation}{\label{eq:kkt-condition-3-gr}}
    \langle \|\boldsymbol{\ell}\|_D^2 \widehat{\boldsymbol{\theta}} + \widehat{\boldsymbol{\ell}}, \widehat{\boldsymbol{r}}\rangle = \sum_{b \in B_{\boldsymbol{\phi}}} \nu_b \langle \mathbf{a}_b, \widehat{\boldsymbol{r}}\rangle.
\end{equation}
As $\widehat{\boldsymbol\ell}$ zeros out every component $b \in B_{\boldsymbol\phi}$ of $\boldsymbol\phi$ with $\langle \mathbf{a}_b, \widehat{\boldsymbol{\ell}} \rangle > 0$, $\lambda'$ is admissible for $\mathrm{gr}_\mathrm{HN}(\boldsymbol\phi)$ if and only if $\langle \mathbf{a}_b, \widehat{\boldsymbol{r}}\rangle \geq 0$ for all $b \in B_{\boldsymbol{\phi}}$ such that $\langle \mathbf{a}_b, \widehat{\boldsymbol{\ell}} \rangle = 0$. By \eqref{eq:kkt-condition-2-gr}, this implies that the right-hand side of \eqref{eq:kkt-condition-3-gr} is non-negative. The left-hand side of \eqref{eq:kkt-condition-3-gr} is equal to the pairing $(\chi_\lambda,\lambda')$ of the character $\chi_\lambda$ in Proposition~\ref{prop:shifted-linearisation}, under the normalisation $\langle \boldsymbol\Theta, \widehat{\boldsymbol{\ell}}\rangle = -1$. Since $\lambda'$ was arbitrary, the Hilbert-Mumford Criterion is satisfied for $\widetilde{\boldsymbol{\theta}}$. This proves (b) and together with the normalisation $\langle \boldsymbol\Theta, \boldsymbol\ell\rangle = -1$ in (a), this finishes the proof of the existence of $\mathrm{HN}(\boldsymbol\phi)$.

Now we prove uniqueness. Suppose $(G^\bullet W,\boldsymbol{\kappa})$ is a $\mathbb{Q}$-filtration admissible for $\boldsymbol{\phi}$ which satisfies (a) and (b), and let $\mathrm{HN}(\boldsymbol\phi) = (F^\bullet W,\boldsymbol{\ell})$. Using the fact that the intersection of any two parabolic subgroups of $G$ contains a maximal torus $T$ \cite[Theorem 17.16]{milne_algebraic_groups}\cite[Corollary 14.13]{borel1991}, we can find a basis for each $W_v$ which is \textit{adapted} to both filtrations $F^\bullet W$ and $G^\bullet W$, i.e., every $F^kW_v$ and $G^kW_v$ is a span of a subset of the basis. Define $\lambda, \mu \in X_*(T)$ to be the rational 1-PSs that induce $(F^\bullet W,\boldsymbol{\ell})$ and $(G^\bullet W,\boldsymbol{\kappa})$ respectively. We know that $\lambda$ and $\mu$ are both normalised and admissible for $\boldsymbol{\phi}$, and since they are both diagonalised by the adapted basis, they commute, and hence are both admissible for $\mathrm{gr}_G(\boldsymbol{\phi})$. By (b), $\mathrm{gr}_G(\boldsymbol{\phi})$ is $\widetilde{\boldsymbol{\theta}}$-semistable, so applying the Hilbert-Mumford Criterion to $\lambda \in X_*(L(\mu))$, we therefore have the inequality
\begin{equation}{\label{eq:kkt-condition-unique-1}}
    \|\boldsymbol{\kappa}\|_D^2\langle \widehat{\boldsymbol{\theta}},\widehat{\boldsymbol{\ell}}\rangle + \langle \widehat{\boldsymbol{\kappa}},\widehat{\boldsymbol{\ell}}\rangle \geq 0
\end{equation}
seen in \eqref{eq:kkt-condition-3-gr}. But since $\langle \widehat{\boldsymbol{\theta}},\widehat{\boldsymbol{\ell}}\rangle = -1$ by (a), \eqref{eq:kkt-condition-unique-1} becomes
\begin{equation}{\label{eq:kkt-condition-unique-2}}
    \langle \widehat{\boldsymbol{\kappa}},\widehat{\boldsymbol{\ell}} - \widehat{\boldsymbol{\kappa}}\rangle \geq 0
\end{equation}
By the definition of $\mathrm{HN}(\boldsymbol{\phi})$, $\lambda$ is a minimiser of the Kempf functional, so
\begin{equation}{\label{eq:kkt-condition-unique-3}}
    \|\widehat{\boldsymbol{\kappa}}\| \geq \|\widehat{\boldsymbol{\ell}}\|.    
\end{equation}
Combining the identity
$$\|\widehat{\boldsymbol{\ell}}\|^2 = \langle \widehat{\boldsymbol{\kappa}} + (\widehat{\boldsymbol{\ell}} - \widehat{\boldsymbol{\kappa}}), \widehat{\boldsymbol{\kappa}} + (\widehat{\boldsymbol{\ell}} - \widehat{\boldsymbol{\kappa}})\rangle = \|\widehat{\boldsymbol{\kappa}}\|^2 + 2\langle \widehat{\boldsymbol{\kappa}}, \widehat{\boldsymbol{\ell}} - \widehat{\boldsymbol{\kappa}}\rangle + \|\widehat{\boldsymbol{\ell}} - \widehat{\boldsymbol{\kappa}}\|^2$$
with \eqref{eq:kkt-condition-unique-2} and \eqref{eq:kkt-condition-unique-3}, we obtain
$$\|\widehat{\boldsymbol{\kappa}}\|^2 \geq \|\widehat{\boldsymbol{\kappa}}\|^2 + 2\langle \widehat{\boldsymbol{\kappa}}, \widehat{\boldsymbol{\ell}} - \widehat{\boldsymbol{\kappa}}\rangle + \|\widehat{\boldsymbol{\ell}} - \widehat{\boldsymbol{\kappa}}\|^2 \geq \|\widehat{\boldsymbol{\kappa}}\|^2 + \|\widehat{\boldsymbol{\ell}} - \widehat{\boldsymbol{\kappa}}\|^2.$$
Thus, $0 \geq \|\widehat{\boldsymbol{\ell}} - \widehat{\boldsymbol{\kappa}}\|^2$ so $\widehat{\boldsymbol{\ell}} = \widehat{\boldsymbol{\kappa}}$, $\lambda = \mu$, and $(G^\bullet W,\boldsymbol{\kappa}) = \mathrm{HN}(\boldsymbol\phi)$.  
\end{proof}

\begin{remark}\label{rem:quivers-hn-weights}
If $\mathcal{H}$ is a quiver, then every row $\mathbf{a}_c$ of $A$ in \eqref{eq:rows-of-A} is of the form $\mathbf{a}_c = \mathbf{e}_j - \mathbf{e}_i$, where $i$ and $j$ correspond to the weight indices of the source and target vertices for $c = (e,I,J)$. Since $c \in \mathcal{C}(\mathrm{gr}_{\mathrm{HN}}(\boldsymbol\phi))$ indexes the components of $\boldsymbol{\phi}$ where $\ell_i = \ell_j$, and since every weight is distinct, we have $i = j$ so $A = 0$. Thus, the formula for $\boldsymbol{p}$ in \eqref{eq:HN-weights-formula-thm} becomes
$$\boldsymbol{p} = -D^{-1}\boldsymbol{\Theta}$$
which, componentwise, is precisely the Harder-Narasimhan slopes
$$\mu_{\boldsymbol{\theta}}(\mathbf{d}^{i} - \mathbf{d}^{i-1}) = -\frac{\sum_{v \in V} \theta_v \cdot d^{(i)}_v}{\sum_{v \in V} d^{(i)}_v}$$
in \eqref{eq:slope-quiver} of the quiver representation $\boldsymbol{\phi}$.
\end{remark}

\printbibliography

\end{document}